\documentclass[12pt]{amsart}
\usepackage{amscd,amsmath,amsthm,amssymb,graphics,color}
\usepackage{pstcol,pst-plot,pst-3d}
\usepackage{lmodern,pst-node}
\usepackage{multicol}
\usepackage{epic,eepic}
\usepackage{amsfonts,amssymb,amscd,amsmath,enumerate,verbatim}

\psset{unit=0.7cm,linewidth=0.8pt,arrowsize=2.5pt 4}

\newpsstyle{fatline}{linewidth=1.5pt}
\newpsstyle{fyp}{fillstyle=solid,fillcolor=verylight}
\definecolor{verylight}{gray}{0.97}
\definecolor{light}{gray}{0.9}
\definecolor{medium}{gray}{0.85}
\definecolor{dark}{gray}{0.6}

\def\NZQ{\Bbb}               
\def\NN{{\NZQ N}}

\def\ZZ{{\NZQ Z}}

\def\FF{{\NZQ F}}

\def\frk{\frak}               

\def\mm{{\frk m}}

\def\Phi{{\frk n}}
\def\Phi{{\frk N}}
\def\cc{{\frk c}}
\def\MI{{\mathcal I}}
\def\MJ{{\mathcal J}}

\def\MQ{{\mathcal Q}}

\def\MT{{\mathcal T}}

\def\MS{{\mathcal S}}
\def\MG{{\mathcal G}}

\def\yb{{\bold y}}

\def\opn#1#2{\def#1{\operatorname{#2}}} 
\opn\chara{char} \opn\length{\ell} \opn\pd{pd} \opn\rk{rk}
\opn\projdim{proj\,dim} \opn\injdim{inj\,dim} \opn\rank{rank}
\opn\depth{depth} \opn\grade{grade} \opn\height{height}\opn\coheight{coheight}
\opn\embdim{emb\,dim} \opn\codim{codim}

\opn\Tr{Tr} \opn\bigrank{big\,rank}
\opn\superheight{superheight}\opn\lcm{lcm}
\opn\trdeg{tr\,deg}
\opn\reg{reg} \opn\lreg{lreg} \opn\ini{in} \opn\lpd{lpd}
\opn\size{size}\opn\bigsize{bigsize}
\opn\cosize{cosize}\opn\bigcosize{bigcosize}
\opn\sdepth{sdepth}\opn\sreg{sreg}
\opn\link{link}\opn\fdepth{fdepth}\opn\type{type}
\opn\GL{GL}
\opn\div{div} \opn\Div{Div} \opn\cl{cl} \opn\Cl{Cl}
\opn\Spec{Spec} \opn\Supp{Supp} \opn\supp{supp} \opn\Sing{Sing}
\opn\Ass{Ass} \opn\Min{Min}\opn\Mon{Mon} \opn\dstab{dstab} \opn\astab{astab}
\opn\Syz{Syz}
\opn\Ann{Ann} \opn\Rad{Rad} \opn\Soc{Soc}
\opn\Im{Im} \opn\Ker{Ker} \opn\Coker{Coker} \opn\Am{Am}
\opn\Hom{Hom} \opn\Tor{Tor} \opn\Ext{Ext} \opn\End{End}
\opn\Aut{Aut} \opn\id{id}

\opn\nat{nat}
\opn\pff{pf}
\opn\Pf{Pf} \opn\GL{GL} \opn\SL{SL} \opn\mod{mod} \opn\ord{ord}
\opn\Gin{Gin} \opn\Hilb{Hilb}\opn\sort{sort}
\opn\Proj{Proj}
\opn\aff{aff} \opn\con{conv} \opn\relint{relint} \opn\st{st}
\opn\lk{lk} \opn\cn{cn} \opn\core{core} \opn\vol{vol}
\opn\link{link} \opn\star{star}\opn\lex{lex}
\opn\gr{gr}

\def\pot#1#2{#1[\kern-0.28ex[#2]\kern-0.28ex]}

\opn\dirlim{\underrightarrow{\lim}}
\opn\inivlim{\underleftarrow{\lim}}
\let\union=\cup
\let\sect=\cap

\let\tensor=\otimes
\let\iso=\cong

\let\to=\rightarrow

\def\Implies{\ifmmode\Longrightarrow \else
        \unskip${}\Longrightarrow{}$\ignorespaces\fi}
\def\implies{\ifmmode\Rightarrow \else
        \unskip${}\Rightarrow{}$\ignorespaces\fi}
\def\iff{\ifmmode\Longleftrightarrow \else
        \unskip${}\Longleftrightarrow{}$\ignorespaces\fi}

\let\:=\colon
\newtheorem{Theorem}{Theorem}[section]
 \newtheorem{Lemma}[Theorem]{Lemma}
 \newtheorem{Corollary}[Theorem]{Corollary}
 \newtheorem{Proposition}[Theorem]{Proposition}
 \newtheorem{Remark}[Theorem]{Remark}
 
 \newtheorem{Example}[Theorem]{Example}
 \newtheorem{Examples}[Theorem]{Examples}
 \newtheorem{Definition}[Theorem]{Definition}
 \newtheorem{Problem}[Theorem]{Problem}

\let\epsilon\varepsilon
\let\kappa=\varkappa
\def\qed{\ifhmode\textqed\fi
      \ifmmode\ifinner\quad\qedsymbol\else\dispqed\fi\fi}
\def\textqed{\unskip\nobreak\penalty50
       \hskip2em\hbox{}\nobreak\hfil\qedsymbol
       \parfillskip=0pt \finalhyphendemerits=0}
\def\dispqed{\rlap{\qquad\qedsymbol}}

\opn\dis{dis}
\def\pnt{{\raise0.5mm\hbox{\large\bf.}}}

\opn\Lex{Lex}

\begin{document}
 \title {Syzygies of Hibi rings}

 \author {Viviana Ene}

\address{Viviana Ene, Faculty of Mathematics and Computer Science, Ovidius University, Bd.\ Mamaia 124,
 900527 Constanta, Romania, and
 \newline
 \indent Simion Stoilow Institute of Mathematics of the Romanian Academy, Research group of the project  ID-PCE-2011-1023,
 P.O.Box 1-764, Bucharest 014700, Romania} \email{vivian@univ-ovidius.ro}

\thanks{The  author was supported by the grant UEFISCDI,  PN-II-ID-PCE- 2011-3-1023.}
\thanks{Part of this survey was presented in a series of three lectures in the 22nd National School on Algebra, September 1-5, 2014, Bucharest, Romania.}

 \begin{abstract}
We survey recent results on resolutions of Hibi rings.
 \end{abstract}

\thanks{}
 \maketitle
\tableofcontents

\section*{Introduction}

Hibi rings and ideals are algebraic objects which arise naturally from combinatorics. They were introduced in 1987 by Hibi in his paper \cite{Hibi}. Hibi rings appear in various combinatorial and algebraic contexts. For example, the coordinate ring of a flag variety for $\GL_n$ is a flat deformation of the Hibi ring on a certain poset known as the Gelfand--Tsetlin poset. Recently, it was observed that several other algebras which arise naturally in representation theory can be described by using Hibi rings.

Let $L$ be a finite distributive lattice. By the well-known theorem of Birkhoff, $L$ is the ideal lattice $\MI(P)$ of its subposet $P$ of join-irreducible elements. Let $P=\{p_1,\ldots,p_n\}$ and let $R=K[t,x_1,\ldots,x_n]$ be the polynomial ring in $n+1$ variables over a field $K.$ The Hibi ring associated with $L$ is the toric ring generated over $K$ by the  monomials $u_{\alpha}=t\prod_{p_i\in \alpha}x_i$ where $\alpha\in L.$ The ring $R[L]$ may be viewed as a standard graded algebra over $K$ if we set $\deg u_{\alpha}=1$ for 
all $\alpha\in L.$

In \cite{Hibi}, Hibi showed that $R[L]$ is an algebra with straightening laws  on $L$ over $K.$ Hence, its defining ideal $I_L$ is generated by the straightening relations of $R[L].$ Let $K[L]$ be the polynomial ring over $K$ in the variables $x_{\alpha}$ with $\alpha\in L.$ Then $I_L\subset K[L]$ is generated by all the binomials $x_\alpha x_\beta-x_{\alpha\cap\beta} x_{\alpha\cup\beta}$ with 
$\alpha,\beta\in L,$ incomparable elements. This is called the Hibi ideal of $L.$ Obviously, $I_L$ is a prime ideal. If $L$ is a chain, then $R[L]=K[L]$, thus any polynomial ring in finitely many variables may be considered as a Hibi ring. Hibi showed that $R[L]$ is a Cohen-Macaulay normal domain. Moreover, he characterized the Gorenstein rings $R[L]$ in terms of the subposet $P$ of $L.$

In the last decades, several properties of Hibi rings and ideals have been investigated. For example,  Gr\"obner bases of Hibi ideals were studied in \cite{AHH, HHNach, Q2} and \cite[Chapter 10]{HH10}.  Other properties, like strongly Koszulness, Koszul filtration, and the divisor class group of a Hibi ring,  were examined in \cite{EHH2,HHN,HHR}.

A number of authors have considered a more general construction. For any finite lattice $L,$ that is, not necessarily distributive, one may consider the graded ideal $I_L=(x_\alpha x_\beta-x_{\alpha\cap\beta} x_{\alpha\cup\beta}: \alpha,\beta\in L, \alpha, \beta \text 
{ incomparable })\subset K[L].$ The quotient ring $K[L]/I_L$ may be viewed as the projective coordinate ring of a projectively embedded 
scheme $V(L)=\Proj(K[L]/I_L)$. As Hibi showed in \cite{Hibi}, $V(L)$ is a toric variety if and only if $L$ is distributive.  The geometric properties of this variety were studied in \cite{W}. For arbitrary lattices, the ideal $I_L$ may be even non radical. However, classes of non-distributive lattices for which $I_L$ is a radical ideal can be identified. Such a class is given in \cite{EHi}. In  paper \cite{EHi}
 it was also shown that the minimal prime ideals of the radical ideal $I_L$ can be characterized in terms of the combinatorics of the lattice $L.$

The notion of Hibi ring associated with a distributive lattice on a poset $P$ was generalized in \cite{EHM}. Generalized Hibi rings and some of their properties are also  surveyed in this paper. 

Several recent works have approached the resolution of Hibi ideals attached to distributive lattices. In this frame, one may of course ask whether the homological invariants of $I_L$ or, even more precise, the graded Betti numbers of $I_L$ may be related to  the combinatorics  
 of $L$ or of its poset $P.$ The projective dimension of $I_L$ and its regularity are already known. But almost nothing is known about the graded Betti numbers of $I_L$ or, equivalently, of $R[L].$ Of course, we would like to have formulas (or at least sharp bounds) for 
the graded Betti numbers in terms of the numerical invariants of $L$ or $P.$ Certainly, this is an interesting area of future research.

In what follows, we present the organization of this survey. In Section~\ref{Hibisection}, after reviewing the necessary background of combinatorics and commutative algebra, we present the construction of  Hibi rings associated with finite distributive lattices as they were introduced by Hibi. We explain in Subsection~\ref{subsectHibi} their structure of algebras with straightening laws. Theorem~\ref{GBHibi} shows that the generators of the Hibi ideal $I_L$ form the reduced Gr\"obner basis of $I_L$ with respect to the reverse lexicographic order induced by a linear order of the variables $x_\alpha$ such that $x_\alpha<x_\beta$ if $\alpha\subset \beta$ in $L.$ As a consequence, it follows that the Hibi ring $R[L]$ is a Cohen-Macaulay domain. Subsection~\ref{subsectHibi} ends with a few comments regarding ideals associated with non-distributive lattices. In Subsection~\ref{canonicalHibi}, we present the combinatorial interpretation 
for the generators of the canonical module $\omega_L$ of $R[L]$. Theorem~\ref{GorensteinHibi} states that $R[L]$ is a Gorenstein ring if and only if the subposet $P$ of join-irreducible elements of $L$ is pure, that is, all its maximal chains have the same length. The last subsection of Section~\ref{Hibisection} presents generalized Hibi rings as they were defined in \cite{EHM}. For any  integer $r\geq 2$ and any finite poset $P,$ we show that the generalized Hibi ring $R_r(P)$ is an algebra with straightening laws on the lattice $\MI_r(P)$
 which consists of the $r$--multichains of $L=\MI(P).$ It then follows that the defining relations of the generalized Hibi ring are classical Hibi relations corresponding to the lattice $\MI_r(P),$ thus $R_r(P)$ is  the Hibi ring associated with $\MI_r(P).$ The poset 
$P^\prime$ of join-irreducible elements in $\MI_r(P)$ turns out to be isomorphic to the cartesian product of $P$ and the set 
$\{1,2,\ldots,r-1\}$ endowed with the natural order; see Theorem~\ref{thm43}. Therefore, the generalized Hibi ring $R_r(P)$ is Gorenstein 
if and only if $P$ is pure and if and only if the Hibi ring $R[\MI(P)]$ is Gorenstein; see Corollary~\ref{corgengor}.

Section~\ref{pseudo} is devoted to level and pseudo--Gorenstein Hibi rings. The notion of pseudo-Gorenstein algebra has been recently introduced in \cite{EHHSara}. In Subsection~\ref{pseudo-algebras}, we recall the necessary definitions and present a characterization of 
pseudo--Gorenstein algebras in Proposition~\ref{echivalent}. A combinatorial characterization of pseudo-Gorenstein Hibi rings is given in Theorem~\ref{classification}. Sufficient conditions for the levelness of Hibi rings were first given by Miyazaki \cite{M}; see Theorem~\ref{Myiazaki} and the remark after it. There are quite simple examples which show that none of those sufficient conditions is  necessary. Later on, in \cite{EHHSara}, a necessary condition for the levelness of Hibi rings was found. We give it in Theorem~\ref{alsoviviana} which states that if $L$ is level, then 
\begin{eqnarray}\label{neccond}
\label{inequality}
\height(x)+\depth(y)  \leq \rank \hat{P}+1
\end{eqnarray}
for all $x,y\in P$ such that $x$ covers $y$. Here $\hat{P}$ denotes the poset $P\union\{-\infty, \infty\}$ with $-\infty<x<\infty$ for all 
$x\in P,$ and, for any element $x\in \hat{P,}$ $\height x $ is the rank of the subposet of $\hat{P}$ which consists of all elements $y\in \hat{P}$ with $y\leq x,$ while $\depth x$ is the rank of the subposet of $\hat{P}$ which consists of all elements $y\in \hat{P}$ with $y\geq x.$
 Condition (\ref{neccond}) is also sufficient at least for a special class of planar lattices that satisfy a regularity condition as it is shown in Theorem~\ref{converse}. Section~\ref{pseudo} ends with a review on level and pseudo-Gorenstein generalized  Hibi rings.

The first subsection of Section~\ref{regsection} presents the formula for the regularity of the Hibi ring. If $L=\MI(P)$ is a distributive lattice, we have
\begin{equation}\label{intro}
\reg R[L]=|P|-\rank P -1.
\end{equation}
As a straightforward consequence of this formula, one gets the characterization of the lattices whose Hibi rings have a linear resolution; see Corollary~\ref{linear resolution}. Moreover, formula (\ref{intro}) allows us to characterize in combinatorial terms 
several classes of Hibi rings with small regularity. Subsection~\ref{regsubsection} ends with a short discussion on planar distributive lattices for which we may identify the regularity of the associated Hibi rings in terms of cyclic sublattices. Subsection~\ref{subsectlinrel} presents planar distributive lattices with the property that their Hibi ideals have linear syzygies, that is, 
$\beta_{1j}(I_L)=0$ 
for $j\geq 4;$ see Theorem~\ref{linrel}. We took the same approach used in \cite{EHH} for determining the linearly related polyomino ideals. Based on the results of this subsection, in the last part of Section~\ref{regsection} we are able to determine all the simple planar distributive lattices $L$ with the property that $R[L]$ has a pure resolution.  

Throughout this survey, we indicated references to the extensive literature on the fundamental notions and results of commutative algebra used in proofs. We assume as well  that the reader has a basic knowledge of Stanley-Reisner theory and Gr\"obner bases. For more information in these areas we recommend \cite[Chapter 5]{BHbook}, \cite{Sta2}, and \cite[Chapter 2]{HH10}, \cite{EH}.

\section{Hibi rings and their Gr\"obner bases}\label{Hibisection}

\subsection{Preliminaries of combinatorics}
\label{combinat}

In this section we review the definitions of the combinatorial objects that will be used throughout these lectures. For a comprehensive treatment and for references to the literature on this subject one may refer to the books of Stanley \cite{Sta1}
 and Birkhoff \cite{B}. 

\begin{Definition}\label{poset}{\em
A {\em partially ordered set} ({\em poset} in brief) is a set $P$ endowed with a {\em partial order} $\leq$, that is, a  relation which is
\begin{itemize}
	\item [(i)] reflexive:  $x\leq x$ for all $x\in P;$
	\item [(ii)] antisymmetric: for any $x,y\in P,$ if $x\leq y$ and $y\leq x,$ then $x=y;$ 
	\item [(iii)] transitive: for any $x,y,z\in P,$ if $x\leq y$ and $y\leq z,$ then $x\leq z.$
\end{itemize}
}
\end{Definition}

We use the notation $x\geq y$ if $y\leq x$ and $x<y$ if $x\leq y$ and $x\neq y.$ If $x\leq y$ or $y\leq x$ we say that 
$x,y$ are {\em comparable} in $P.$ Otherwise, $x,y$ are {\em incomparable.}

All the posets in these lectures are assumed to be finite.

\begin{Examples}\label{basics}{\em 

\begin{itemize}
	\item  [1.] Let $n\geq 1$ be an integer and  $[n]=\{1,2,\ldots,n\}$. This set is a poset with the natural order of integers. Any two elements of $[n]$ are comparable.
 \item  [2.] Let $B_n=2^{[n]}$ be the power set of $[n]$. $B_n$ is a partially ordered set with the inclusion. Obviously, not any 
two subsets of $[n]$ are comparable with respect to inclusion.
\item [3.] Let $n\geq 2$ be an integer and $D_n$ the set of all divisors of $n$. $D_n$ is partially ordered with respect to  divisibility.
\end{itemize}
}
\end{Examples}

Any finite poset $P$ is completely determined by its cover relations which are encoded in the  Hasse diagram of $P.$ We say that $y$ {\em covers} $x$ if $y>x$ and there is no $z\in P$ with $y>z>x.$ In this case we write $y\gtrdot x.$ The Hasse diagram of $P$ is a graph whose vertices are the elements of $P$ and the edges are the cover relations of $P.$

For example, in Figure~\ref{Boole}, the Hasse diagrams of $B_3$ and $D_{12}$ are displayed.

\begin{figure}[hbt]
\begin{center}
\psset{unit=0.5cm}
\begin{pspicture}(4,-1)(4,7)
\rput(-1,0){
\rput(0,2){$\bullet$}
\rput(0,4){$\bullet$}
\rput(2,0){$\bullet$}
\rput(2,2){$\bullet$}
\rput(4,2){$\bullet$}
\rput(4,4){$\bullet$}
\rput(2,4){$\bullet$}
\rput(2,6){$\bullet$}
\pspolygon(0,2)(2,0)(4,2)(2,4)
\pspolygon(0,4)(2,2)(4,4)(2,6)
\psline(0,2)(0,4)
\psline(2,0)(2,2)
\psline(4,2)(4,4)
\psline(2,4)(2,6)
\rput(2,-1){$B_3$}
}
\rput(5,0){
\rput(0,2){$\bullet$}
\rput(2,0){$\bullet$}
\rput(4,2){$\bullet$}
\rput(2,4){$\bullet$}
\rput(6,4){$\bullet$}
\rput(4,6){$\bullet$}
\pspolygon(0,2)(2,0)(6,4)(4,6)
\psline(4,2)(2,4)

\rput(2,-1){$D_{12}$}
}
\end{pspicture}
\end{center}
\caption{}\label{Boole}
\end{figure}

In Figure~\ref{example1}, we have the Hasse diagram of a poset $P$ with $5$ elements, $x,y,z,t,u$ 
with $z\gtrdot x, z\gtrdot y, t\gtrdot y, t\gtrdot x, u\gtrdot t.$

\begin{figure}[hbt]
\begin{center}
\psset{unit=0.8cm}
\begin{pspicture}(-1,0)(4,2)

\rput(0,0){$\bullet$}
\rput(0,1){$\bullet$}
\rput(2,0){$\bullet$}
\rput(2,1){$\bullet$}
\rput(2,2){$\bullet$}
\psline(0,0)(0,1)
\psline(2,0)(0,1)
\psline(2,0)(2,1)
\psline(2,1)(2,2)
\psline(0,0)(2,1)

\rput(0,-0.3){$x$}
\rput(2.3,-0.3){$y$}
\rput(0,1.3){$z$}
\rput(2.3,2){$u$}
\rput(2.3,1){$t$}
\end{pspicture}
\end{center}
\caption{The Hasse diagram}\label{example1}
\end{figure}
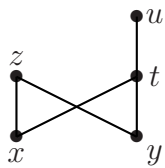

A {\em subposet} of $P$ is a subset $Q$ endowed with a partial order such that, for $x,y\in Q,$ we have 
$x\leq y$ in $Q$ if and only if $x\leq y$ in $P.$

For example, for the poset $P$ displayed in Figure~\ref{example1}, the poset $Q$ displayed in Figure~\ref{subposet} (a) 
is a subposet of $P$ while the poset $Q^\prime$ displayed in Figure~\ref{subposet} (b) is not.

\begin{figure}[hbt]
\begin{center}
\psset{unit=0.8cm}
\begin{pspicture}(-1,-1)(4,2)
\rput(-1,0){
\rput(0,0){$\bullet$}
\rput(0,1){$\bullet$}
\rput(2,0){$\bullet$}
\rput(2,1){$\bullet$}
\psline(0,0)(0,1)
\psline(2,0)(0,1)
\psline(2,0)(2,1)
\psline(0,0)(2,1)
\rput(1,-0.8){(a)}
\rput(0,-0.3){$x$}
\rput(2,-0.3){$y$}
\rput(0,1.3){$z$}
\rput(2,1.3){$t$}
}
\rput(3,0){
\rput(0,0){$\bullet$}
\rput(0,1){$\bullet$}
\rput(2,0){$\bullet$}
\rput(2,1){$\bullet$}
\psline(0,0)(0,1)
\psline(2,0)(2,1)
\psline(0,0)(2,1)
\rput(1,-0.8){(b)}
\rput(0,-0.3){$x$}
\rput(2,-0.3){$y$}
\rput(0,1.3){$z$}
\rput(2,1.3){$t$}
}
\end{pspicture}
\end{center}
\caption{}\label{subposet}
\end{figure}

Let $P$ be a poset and $x,y\in P$ with $x\leq y.$ The set 
\[
[x,y]=\{z\in P: x\leq z\leq y\}
\]
is called a ({\em closed}) {\em interval} in $P.$ Obviously, any interval of $P$ is a subposet of $P.$ For example, for the poset displayed in Figure~\ref{example1}, we have $[x,z]=\{x,z\}$ and $[y,u]=\{y,t,u\}.$

\begin{Definition}\label{morphism}{\em
Let $P$ and $Q$ be two posets. An order preserving map $f:P\to Q$ is called a {\em morphism} of posets. The posets $P$ and $Q$
 are called {\em isomorphic} if there exists a bijection $f:P\to Q$ which is a morphism of posets with the property that 
$f^{-1}$ is a morphism as well.
}
\end{Definition}

A partial order on $P$ is called a {\em total order} or {\em linear order}  if any two elements of $P$ are comparable, that is, for any $x,y\in P,$ we have either $x\leq y$ or $y\leq x.$ If $\leq $ is a total order on $P,$ we call $P$ a {\em totally ordered } set or {\em chain}. $P$ is an {\em antichain} or {\em clutter} if any two different elements of $P$ are incomparable. 

\begin{figure}[hbt]
\begin{center}
\psset{unit=0.8cm}
\begin{pspicture}(-1,-1)(4,3)
\rput(-2,0){
\rput(0,0){$\bullet$}
\rput(0,1){$\bullet$}
\rput(0,2){$\bullet$}
\rput(0,3){$\bullet$}
\psline(0,0)(0,3)
\rput(0,-0.8){Chain}
}
\rput(3,0){
\rput(0,0){$\bullet$}
\rput(1,0){$\bullet$}
\rput(2,0){$\bullet$}
\rput(3,0){$\bullet$}
\rput(4,0){$\bullet$}
\rput(2,-0.8){Antichain}
}
\end{pspicture}
\end{center}
\caption{}\label{Chain}
\end{figure}

Given the poset $P,$  a chain in $P$ is a  subposet $C$ of $P$ which is totally ordered. If $C$ is a chain of $P,$ 
$\ell(C)=|C|-1$ is the {\em length} of $C$. A chain $C: x_0<x_1<\cdots <x_r$ in $P$ is called {\em saturated} if 
$x_{i+1}\gtrdot x_i$ for all $i.$  

\begin{Definition}\label{graded}{\em 
Let $P$ be a poset. The {\em rank} of $P$ is 
\[
\rank P=\max\{\ell(C): C \text{ is a chain of }P\}.
\]
If every maximal chain of $P$ has the same length, then $P$ is called {\em graded} or {\em pure}.
}
\end{Definition}

For example, the posets of Figure~\ref{Boole} are graded of rank $3$ while the poset of Figure~\ref{example1} is not graded.

A {\em minimal} element of a poset $P$ is an element $x\in P$ such that, for any $y\in P$, if $y\leq x$ then $y=x.$
In other words, if $y,x$ are comparable, then $y\geq x.$ By dualizing the above conditions, that is, taking $\geq$ instead of $\leq,$ we define the {\em maximal} elements of $P.$ For example, in the poset displayed in Figure~\ref{example1} there are two minimal elements, namely $x,y$, and two maximal elements, $z,u.$

For a poset $P,$ $\hat{P}$ denotes the poset $P\cup\{-\infty,\infty\}$ where, for $x,y\in P,$ $x\leq y$ in $\hat{P}$ if and only if $x\leq y$ in $P$ and $-\infty<x<\infty$ for all $x\in P.$ For example, the poset $\hat{P}$ for the poset of Figure~\ref{example1} is displayed in Figure~\ref{hat}.

\begin{figure}[hbt]
\begin{center}
\psset{unit=0.8cm}
\begin{pspicture}(-1,-1)(4,3)

\rput(0,0){$\bullet$}
\rput(0,1){$\bullet$}
\rput(2,0){$\bullet$}
\rput(2,1){$\bullet$}
\rput(2,2){$\bullet$}
\rput(1,3){$\bullet$}
\rput(1,-1){$\bullet$}
\psline(0,0)(0,1)
\psline(2,0)(0,1)
\psline(2,0)(2,1)
\psline(2,1)(2,2)
\psline(0,0)(2,1)
\psline(0,1)(1,3)
\psline(2,2)(1,3)
\psline(0,0)(1,-1)
\psline(2,0)(1,-1)

\rput(-0.3,-0.3){$x$}
\rput(2.3,-0.3){$y$}
\rput(-0.3,1.3){$z$}
\rput(2.3,2){$u$}
\rput(2.3,1){$t$}
\rput(1,3.3){$\infty$}
\rput(1,-1.3){$-\infty$}
\end{pspicture}
\end{center}
\caption{$\hat{P}$}\label{hat}
\end{figure}

Clearly, if $\hat{P}$ is graded if and only if $P$ is graded.

For a graded poset of rank $n,$ one considers the {\em rank function} $\rho: P\to \{0,1,\ldots,n\}$ defined as follows:
\begin{itemize}
	\item [{}] $\rho(x)=0$ for any minimal element of $P;$
	\item [{}] $\rho(y)=\rho(x)+1$ for $y\gtrdot x$ in $P.$
\end{itemize}

If $\rho(x)=i,$ we say that $\rank x=i.$

\begin{Examples}\label{exrank}{\em 
1. Let $B_n$ be the Boolean poset on the set $[n]$. Then $B_n$ is graded of rank $n$ and, for $x\in B_n,$ $\rank x=|x|.$

2. Let $n\geq 2$ be an integer and $D_n$ the poset of the divisors of $n.$ The poset $D_n$ is graded of rank equal to the number of the prime divisors of $n$ and, for $x|n,$ $\rank x$ is equal to the number of the prime divisor of $x$ (in each case counted with multiplicity).
}
\end{Examples}

\subsubsection{Operations on posets} 
{\bf 1. Direct sums.} Let $P,Q$ be two posets on disjoint sets. The {\em direct sum} of $P$ and $Q$ is the poset $P+Q$ on 
the set $P\cup Q$ with the order defined as follows: $x\leq y$ in $P+Q$ if either $x,y\in P$ and $x\leq y$ in $P$ or 
$x,y\in Q$ and $x\leq y$ in $Q.$ 
A poset $P$ which can be written as a direct sum of to subposets is called {\em disconnected}. Otherwise, $P$ is {\em connected}.

{\bf 2. Ordinal sum.} The {\em ordinal sum} $P\oplus Q$ of the disjoint posets $P,Q$ is the poset on the set $P\cup Q$ with 
the following order. If $x,y\in P\oplus Q,$ then $x\leq y$ if either $x,y\in P$ and $x\leq y$ in $P$ or 
$x,y\in Q$ and $x\leq y$ in $Q$ or $x\in P$ and $y\in Q.$

\begin{Example}{\em In Figure~\ref{ordinal sum} the ordinal sum of two posets is displayed. }

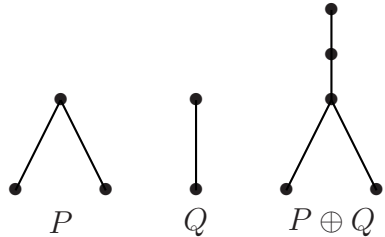
\begin{figure}[hbt]
\begin{center}
\psset{unit=0.6cm}
\begin{pspicture}(7,-1)(4,5)

\rput(0,0){$\bullet$}
\rput(1,2){$\bullet$}
\rput(2,0){$\bullet$}
\rput(4,0){$\bullet$}
\rput(4,2){$\bullet$}
\rput(6,0){$\bullet$}
\rput(8,0){$\bullet$}
\rput(7,2){$\bullet$}
\rput(7,3){$\bullet$}
\rput(7,4){$\bullet$}
\psline(0,0)(1,2)
\psline(2,0)(1,2)
\psline(4,0)(4,2)
\psline(6,0)(7,2)
\psline(8,0)(7,2)
\psline(7,2)(7,4)

\rput(1,-0.7){$P$}
\rput(4,-0.7){$Q$}
\rput(7,-0.7){$P\oplus Q$}
\end{pspicture}
\end{center}
\caption{Ordinal sum}\label{ordinal sum}
\end{figure}
\end{Example}

{\bf 3. Cartesian product.} Let $P$ and $Q$ be two posets. The cartesian product of $P$ and $Q$ is the poset $P\times Q$ on 
the set $P\times Q$ such that $(x,y)\leq (x^\prime,y^\prime)$ in $P\times Q$ if $x\leq x^\prime$ in $P$ and $y\leq y^\prime$ in $Q.$

\begin{Example}{\em Figure~\ref{cproduct} shows a cartesian product of two posets.}
\begin{figure}[hbt]
\begin{center}
\psset{unit=0.6cm}
\begin{pspicture}(8,-1)(4,5)

\rput(0,0){$\bullet$}
\rput(1,2){$\bullet$}
\rput(2,0){$\bullet$}
\rput(4,0){$\bullet$}
\rput(4,2){$\bullet$}
\rput(6,0){$\bullet$}
\rput(8,0){$\bullet$}
\rput(7,2){$\bullet$}
\rput(6,2){$\bullet$}
\rput(8,2){$\bullet$}
\rput(7,4){$\bullet$}
\psline(0,0)(1,2)
\psline(2,0)(1,2)
\psline(4,0)(4,2)
\psline(6,0)(7,2)
\psline(8,0)(7,2)
\psline(7,2)(7,4)
\psline(6,2)(7,4)
\psline(6,0)(6,2)
\psline(8,0)(8,2)
\psline(8,2)(7,4)

\rput(1,-0.7){$P$}
\rput(4,-0.7){$Q$}
\rput(7,-0.7){$P\times Q$}
\end{pspicture}
\end{center}
\caption{Cartesian product}\label{cproduct}
\end{figure}
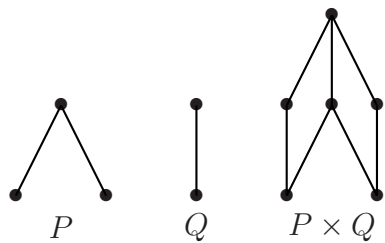
\end{Example}

{\bf 4. The dual poset.} Let $P$ be a poset. The {\em dual} of $P$ is the poset $P^\ast$ on the same set as $P$ such that 
$x\leq y$ in $P^\ast$ if and only if $x\geq y$ in $P.$ If $P$ and $P^\ast$ are isomorphic, then $P$ is called {\em self-dual}.

\begin{Example} {\em Figure~\ref{dual} and Figure~\ref{sdual} show the dual of a poset and a self-dual poset.}

\begin{figure}[hbt]
\begin{center}
\psset{unit=0.6cm}
\begin{pspicture}(3,-1)(4,5)

\rput(0,0){$\bullet$}
\rput(1,2){$\bullet$}
\rput(2,0){$\bullet$}
\rput(1,4){$\bullet$}
\rput(5,0){$\bullet$}
\rput(5,2){$\bullet$}
\rput(4,4){$\bullet$}
\rput(6,4){$\bullet$}

\psline(0,0)(1,2)
\psline(2,0)(1,2)
\psline(1,2)(1,4)
\psline(5,0)(5,2)
\psline(5,2)(4,4)
\psline(5,2)(6,4)

\rput(1,-0.7){$P$}
\rput(5,-0.7){$P^\ast$}
\end{pspicture}
\end{center}
\caption{The dual poset}\label{dual}
\end{figure}
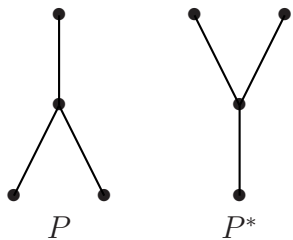

\begin{figure}[hbt]
\begin{center}
\psset{unit=0.6cm}
\begin{pspicture}(3,-1)(4,5)
\rput(0,0){
\rput(0,0){$\bullet$}
\rput(-2,2){$\bullet$}
\rput(2,2){$\bullet$}
\rput(0,4){$\bullet$}

\pspolygon(0,0)(-2,2)(0,4)(2,2)
\rput(0,-0.7){$P$}
}
\rput(5,0){
\rput(0,0){$\bullet$}
\rput(-2,2){$\bullet$}
\rput(2,2){$\bullet$}
\rput(0,4){$\bullet$}

\pspolygon(0,0)(-2,2)(0,4)(2,2)
\rput(0,-0.7){$P^\ast$}
}
\end{pspicture}
\end{center}
\caption{Self-dual poset}\label{sdual}
\end{figure}
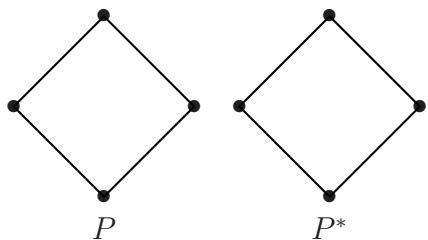
\end{Example}

\subsubsection{Lattices} Let $P$ be a poset and $x,y\in P.$ An {\em upper bound} of $x,y$ is an element $z\in P$ such that 
$z\geq x$ and $z\geq y$. If the set $\{z\in P: z \text{ is an upper bound of } x \text{ and }y\}$ has a least element, this is obviously unique,  is called the {\em join}  of $x$ and $y$, and it is denoted $x\vee y.$ By duality, one defines 
the {\em meet} $x\wedge y$ of two elements $x,y$ in a poset. 

\begin{Definition}{\em 
A lattice $L$ is a poset with the property that for any $x,y\in L,$ $x\vee y$ and $x\wedge y$ exist.
}
\end{Definition}

It is easily seen that if $L$ and $L^\prime$ are lattices, then so are $L^\ast,$ $L\oplus L^\prime$, and $L\times L^\prime.$

\begin{Example}{\em 
$B_n$ and $D_n$ are lattices.
}
\end{Example}
 All the lattices considered in these lectures are finite. Unless otherwise stated, by a lattice we mean a finite lattice. 
Clearly, any lattice has a minimum and a maximum. 

A sublattice of $L$ is a subposet $L^\prime$ of $L$ with the property that for any $x,y\in L^\prime,$ $x\vee y\in L^\prime$ 
and $x\wedge y\in L^\prime.$

\begin{Proposition}\cite[Chapter 3]{Sta1}
Let $L$ be a lattice. The following conditions are equivalent:
\begin{itemize}
	\item [(i)] $L$ is graded and its rank function satisfies $\rho(x)+\rho(y)\geq \rho(x\wedge y)+\rho(x\vee y)$ for all 
	$x,y\in L.$
	\item [(ii)] If $x$ and $y$ cover $x\wedge y$, then $x\vee y$ covers $x$ and $y.$
\end{itemize}
\end{Proposition}

\begin{Definition}{\em 
A lattice $L$ is called {\em modular} if it is graded and its rank function satisfies 
$\rho(x)+\rho(y)= \rho(x\wedge y)+\rho(x\vee y)$ for all 
	$x,y\in L.$
}
\end{Definition}

The following proposition characterizes the modular lattices. For the proof one may consult \cite{B}.

\begin{Proposition}\label{modular}
Let $L$ be a lattice. The following conditions are equivalent:
\begin{itemize}
	\item [(i)] $L$ is a modular lattice;
	\item [(ii)] For all $x,y,z\in L$ such that $x\leq z,$ we have $x\vee(y\wedge z)=(x\vee y)\wedge z.$
	\item [(iii)] $L$ has no sublattice isomorphic to the pentagon lattice of Figure~\ref{pentagon}.
\end{itemize}
\end{Proposition}

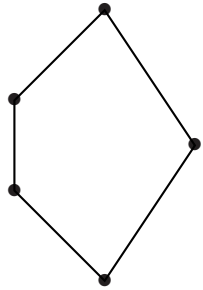
\begin{figure}[hbt]
\begin{center}
\psset{unit=0.6cm}
\begin{pspicture}(-2,0)(4,6)

\rput(0,0){$\bullet$}
\rput(-2,2){$\bullet$}
\rput(-2,4){$\bullet$}
\rput(0,6){$\bullet$}
\rput(2,3){$\bullet$}
\pspolygon(0,0)(-2,2)(-2,4)(0,6)(2,3)

\end{pspicture}
\end{center}
\caption{Pentagon}\label{pentagon}
\end{figure}

\subsubsection{Distributive lattices} In these lectures we are mainly interested in distributive lattices.

\begin{Definition}\label{distrib}{\em
Let $L$ be a lattice. $L$ is called {\em distributive} if satisfies one of the equivalent conditions:
\begin{itemize}
	\item [(i)] for any $x,y,z\in L,$ $x\vee(y\wedge z)=(x\vee y)\wedge(x\vee z)$;
	\item [(ii)] for any $x,y,z\in L,$ $x\wedge(y\vee z)=(x\wedge y)\vee(x\wedge z).$
\end{itemize}
}
\end{Definition} 

The lattices $B_n$ and $D_n$ are distributive while the lattices displayed in Figure~\ref{pentagon} and Figure~\ref{nondis} are not.

\begin{figure}[hbt]
\begin{center}
\psset{unit=0.6cm}
\begin{pspicture}(-2,0)(4,6)
\rput(-2,0){
\rput(0,0){$\bullet$}
\rput(-2,2){$\bullet$}
\rput(-2,4){$\bullet$}
\rput(0,6){$\bullet$}
\rput(2,3){$\bullet$}
\rput(0,3){$\bullet$}
\pspolygon(0,0)(-2,2)(-2,4)(0,6)(2,3)
\psline(0,0)(0,3)
\psline(0,3)(-2,4)
}
\rput(4,0){
\rput(0,0){$\bullet$}
\rput(-2,2){$\bullet$}
\rput(0,4){$\bullet$}
\rput(0,2){$\bullet$}
\rput(2,2){$\bullet$}
\pspolygon(0,0)(-2,2)(0,4)(2,2)
\psline(0,0)(0,4)
}
\end{pspicture}
\end{center}
\caption{Non-distributive lattices}\label{nondis}
\end{figure}
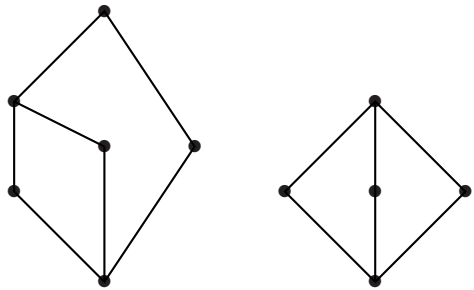

A famous theorem of Birkhoff \cite{B} states that every distributive lattice $L$ is the lattice of the order ideals of a certain suposet $P$ of $L.$

A subset $\alpha$ of a poset $P$ is called an {\em order ideal} or {\em poset ideal} if it satisfies the following condition: 
for any $x\in \alpha$ and $y\in P,$ if $y\leq x,$ then $y\in \alpha.$ The set of all order ideals of $P$ is denoted $\MI(P).$
The union and intersection of two order ideals are obviously order ideals. Therefore, $\MI(P)$ is a distributive lattice with the union and intersection. 

Given a lattice $L,$ an element $x\in L$ is called {\em join-irreducible} if $x\neq \min L$ and whenever $x=y\vee z$ for some 
$y,z\in L$, we have either $x=y$ or $x=z.$

\begin{Theorem}[Birkhoff]\label{Birkoff}
Let $L$ be a distributive lattice and $P$ its subposet of join-irreducible elements. Then $L$ is isomorphic to $\MI(P).$
\end{Theorem}

In the following figure we illustrate Birkhoff's theorem.

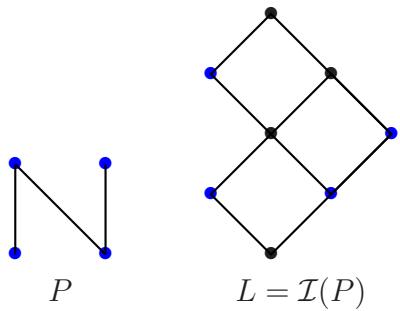
\begin{figure}[hbt]
\begin{center}
\psset{unit=0.6cm}
\begin{pspicture}(-2,-1)(4,6)
\rput(-3,0){
\rput(0,0){\color{blue}$\bullet$}
\rput(0,2){\color{blue}$\bullet$}
\rput(2,0){\color{blue}$\bullet$}
\rput(2,2){\color{blue}$\bullet$}
\psline(0,0)(0,2)
\psline(0,2)(2,0)
\psline(2,0)(2,2)
\rput(1,-0.8){$P$}
}
\psset{unit=0.4cm}
\rput(4,0){
\rput(0,0){$\bullet$}
\rput(-2,2){\color{blue}$\bullet$}
\rput(0,4){$\bullet$}
\rput(2,2){\color{blue}$\bullet$}
\rput(4,4){\color{blue}$\bullet$}
\rput(-2,6){\color{blue}$\bullet$}
\rput(2,6){$\bullet$}
\rput(0,8){$\bullet$}
\pspolygon(0,0)(-2,2)(2,6)(4,4)
\pspolygon(2,2)(-2,6)(0,8)(4,4)
\rput(1,-1.3){$L=\MI(P)$}
}
\end{pspicture}
\end{center}
\caption{Birkhoff's theorem}\label{Birhoff_fig}
\end{figure}

\subsection{Preliminaries of commutative algebra}
\label{comalgebra}

In this subsection we review basic facts about minimal free resolutions and canonical modules which will be needed in what follows. 
 We refer the reader to the book of Stanley \cite{Sta2}, the survey of Herzog \cite{H}, and Chapter $4$ in \cite{EH} for more information.

\subsubsection{Minimal graded free resolutions}
Let $S=K[x_1,\ldots,x_n]$ be the polynomial ring in $n$ variables over a field $K.$ The ring $S$ is graded with the usual grading $S=\oplus_{d\geq 0} S_d$ where $S_d$ is the $K$--vector space generated by all the monomials of $S$ of degree $d.$
The ideal $\mm=S_+=\oplus_{d>0}S_d=(x_1,\ldots,x_n)$ is the unique graded maximal ideal of $S.$

A {\em graded $S$--module}  $M$ has a decomposition $M=\oplus_{n\in \ZZ}M_n$ as a vector space over $K$ with the property that $S_d M_n\subset M_{n+d}$ for all $d,n.$ Most often, we will work with graded modules of the form $S/I$ where $I$ is  a 
graded ideal of $S.$ A graded $K$--algebra of the form $R=S/I$ where $I$ is a graded ideal of $S$ is called a 
{\em standard graded algebra}.

All the graded $S$--modules considered in this paper are finitely generated. Obviously, if $M$ is a finitely generated 
graded $S$--module, then there exists $m\in \ZZ$ such that $M_n=0$ for all $n<m.$

Let $M$ be a graded $S$--module. The function $H(M,-):\ZZ\to \NN$ given by $H(M,n)=\dim_K M_n,$ for all $n,$ is called 
the {\em Hilbert function} of $M.$ The generating series of this function, $H_M(t)=\sum_{n\in \ZZ}H(M,n)t^n,$ is called the 
{\em Hilbert series} of $M.$ For example, the Hilbert function of $S$ is $H(S,d)={n+d-1\choose n}$ and the Hilbert series is 
$H_S(t)=1/(1-t)^n.$

If $M$ is a graded $S$--module and $a$ is an integer, then $M(a)$ is the graded module whose degree $n$ component is 
$(M(a))_n=M_{a+n}$ for all $n.$ By the definition of the Hilbert series, it obviously follows that $H_{M(a)}(t)=t^{-a}H_M(t).$

\begin{Definition}{\em
A {\em graded free resolution}  of the finitely generated graded $S$--module $M$ is an exact sequence 
\[
\\F_{\bullet}:\hspace{0.7cm} 0\to F_p\stackrel{\varphi_p}{\rightarrow}F_{p-1}\stackrel{\varphi_{p-1}}{\rightarrow}
\cdots \stackrel{\varphi_2}{\rightarrow} F_1\stackrel{\varphi_1}{\rightarrow}F_0\stackrel{\varphi_0}{\rightarrow} M\to 0
\] where $F_i$ are free $S$--modules of finite rank and the maps $\varphi_i: F_i\to F_{i-1}$ preserve the degrees, that is, they are graded maps.
}
\end{Definition}

The modules $F_i$ are of the form $F_i=\oplus_{j\in \ZZ} S(-j)^{b_{ij}}$ for all $i.$ 

The resolution $\FF_{\bullet}$ is called {\em minimal} if $\Im \varphi_i\subset \mm F_{i-1}$ for $i\geq 1.$ This is equivalent to saying that all the matrices representing the maps $\varphi_i$ in the resolution have all the entries in $\mm.$ By the Hilbert's Syzygy Theorem (see, for example, \cite[Theorem 4.18]{EH}) it follows that $p\leq n$ if $\FF_{\bullet}$ is minimal.

Any two minimal graded free resolutions of $M$ are isomorphic; see \cite[Theorem 4.25]{EH}. Hence, if $\FF_{\bullet}$
 is minimal and $F_i=\oplus_{j\in \ZZ} S(-j)^{\beta_{ij}}$, then the numbers $\beta_{ij}=\beta_{ij}(M)$ are called 
{\em the graded Betti numbers} of $M.$ The numbers $\beta_i=\beta_i(M)=\sum_j \beta_{ij}$ are called {\em the total Betti numbers} of $M.$ We have the following formulas for $\beta_i(M)$ and $\beta_{ij}(M):$
\[
\beta_i(M)=\dim_K\Tor_i^S(M,K) \text { and } \beta_{ij}(M)=\dim_K\Tor_i^S(M,K)_j.
\]

The following data can be read from the minimal graded free resolution of $M$. The {\em projective dimension} of $M$  is defined as
\[
\projdim M=\max\{i: \beta_{ij}\neq 0 \text{ for some }j\}.
\]  

The {\em regularity} of $M$ is given by 
\[
\reg M=\max\{j-i: \beta_{ij}\neq 0\}.
\]

The graded Betti numbers of $M$ are usually  displayed in the so-called {\em Betti diagram} of $M;$ see Figure~\ref{Betti diagram}.

\begin{figure}[hbt]
\begin{center}
\psset{unit=1cm}
\begin{pspicture}(-0.8,1)(6,5)
 \psline(0.5,1)(0.5,5)
 \psline(0,4.5)(5,4.5)
 \psline[linestyle=dashed](0.5,1.5)(2.5,1.5)
 \psline[linestyle=dashed](2.5,1.5)(2.5,2)
 \psline[linestyle=dashed](2.5,2)(3.5,2)
 \psline[linestyle=dashed](3.5,2)(3.5,3.5)
 \psline[linestyle=dashed](3.5,3.5)(4,3.5)
 \psline[linestyle=dashed](4,3.5)(4,4.5)
 \psline[linestyle=dotted](0.5,3)(2.5,3)
 \psline[linestyle=dotted](2.5,3)(2.5,4.5)
 \rput(0.3,3){{$j$}}
 \rput(2.5,4.8){{$i$}}
 \rput(2.7, 2.7){{$\beta_{ii+j}$}}
 \rput(2.5,1.5){{$\bullet$}}
 \rput(3.5,2){{$\bullet$}}
 \rput(4,3.5){{$\bullet$}}
 \rput(0, 1.5){$\reg$}
 \rput(4,4.8){$\projdim$}
\end{pspicture}
\end{center}
\caption{Betti diagram}\label{Betti diagram}
\end{figure}
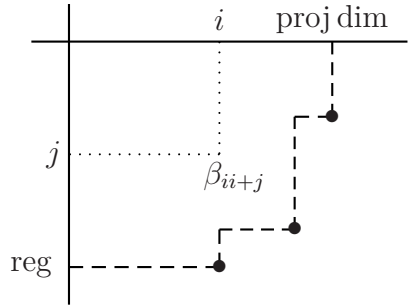

\begin{Definition}{\em 
Let $M$ be a finitely generated graded $S$--module and $\{g_1,\ldots,g_m\}$ a minimal system of homogeneous generators of $M.$
The module $M$ has a {\em $d$--linear resolution} if $\deg g_i=d$ for $1\leq i\leq m$ and $\beta_{ij}(M)=0$ for $j\neq i+d.$ In other words, 
we have $\beta_i(M)=\beta_{i i+d}(M)$ for all $i.$
}
\end{Definition}

Hence, $M$ has a $d$--linear resolution if the minimal graded free resolution is of the form:
\[
0\to S(-d-p)^{\beta_p}\to \cdots \to S(-d-1)^{\beta_1}\to S(-d)^{\beta_0}\to M\to 0.
\]
This is equivalent to saying that all the minimal homogeneous generators have degree $d$ and all the maps in the minimal graded free resolution have linear form entries.

\begin{Definition}{\em 
Let $M$ be a finitely generated graded $S$--module and $\{g_1,\ldots,g_m\}$ a minimal system of homogeneous generators of $M.$
The module $M$ has a {\em pure resolution} if its minimal graded free resolution is of the form
\[
0\to S(-d_p)^{\beta_p}\to \cdots \to S(-d_1)^{\beta_1}\to S(-d_0)^{\beta_0}\to M\to 0.
\] for some integers $0<d_0<d_1<\cdots<d_p.$
}
\end{Definition}

\subsubsection{Cohen-Macaulay modules and canonical modules} 
\label{subsectcanonical}

\begin{Definition}{\em 
Let $M$ be a graded finitely generated $S$--module. A sequence of homogeneous elements $\theta_1,\ldots, \theta_r\in \mm$ is called 
and {\em $M$--sequence} if $\theta_i$ is regular on $M/(\theta_1,\ldots,\theta_{i-1})M$ for all $i$ which means that, for any $i,$ 
the multiplication map $\theta_i:M/(\theta_1,\ldots,\theta_{i-1})M\to M/(\theta_1,\ldots,\theta_{i-1})M$ is injective.
}
\end{Definition}

The length of the longest $M$--sequence of homogeneous elements is called the {\em depth} of $M.$ The Auslander-Buchsbaum Theorem 
\cite[Theorem 1.3.3]{BHbook} states that $$\depth M=n-\projdim M.$$ 

In general, one has $\depth M \leq \dim M;$ see 
\cite[Proposition 1.2.12]{BHbook}. The equality case is very important in commutative algebra.  A finitely generated graded $S$--module 
$M$ is called {\em Cohen-Macaulay} if $\depth M=\dim M.$ 
 
Let $R=S/I$ be a standard graded Cohen-Macaulay algebra of dimension  $d$ with the minimal graded free $S$--resolution 
\[
\FF_{\bullet}: \hspace{0.7cm} 0\to F_{n-d}\stackrel{\varphi_{n-d}}{\rightarrow} F_{n-d-1}\to
\cdots \to F_1\stackrel{\varphi_{1}}{\rightarrow} F_0\stackrel{\varphi_{0}}{\rightarrow} R\to 0.
\]

A finite graded $S$--module $\omega_R$ is the {\em canonical module} of $R$ if 
\[
\Ext_S^i(S/\mm,\omega_R)\cong\left\{
\begin{array}{ll}
	0,& i\neq d\\
	S/\mm, & i=d
\end{array}.\right.
\]

\begin{Example}{\em 
The Koszul complex of $x_1,\ldots,x_n$ gives the minimal graded free resolution of $K=S/\mm$. The last non-zero module in the Koszul complex is $F_n=S(-n).$ Then 
$\Ext_S^i(S/\mm,S)=0$ for $i\neq n$ and $\Ext_S^n(S/\mm,S)=(S/\mm)(n).$ Hence, the canonical module of $S$ is $\omega_S=S(-n).$
}
\end{Example}

Let $R=S/I$ be a Cohen-Macaulay standard graded $K$--algebra and $\FF_{\bullet}$  its minimal graded free resolution. Then the sequence 
\[
0\to F_0^\ast \stackrel{\varphi_1^\ast}{\rightarrow}F_1^\ast\to\cdots \stackrel{\varphi_{n-d}^\ast}{\rightarrow} F_{n-d}^\ast \to \omega_R\to 0
\]
is the minimal graded free resolution of $\omega_R.$ Here $F_i^\ast$ denotes the dual of $F_i$ and $\varphi_i^\ast$ the dual of $\varphi_i$ for all $i.$ 
Therefore, we have
\[
\beta_i(\omega_R)=\beta_{n-d-i} (R) \text{ for all }i.
\] In particular, $\beta_{n-d}(R)$ is equal to the minimal  number of homogeneous generators of $\omega_R.$ The Betti number $\beta_{n-d}(R)$ is called the 
{\em type} of $R$ and it is denoted $\type(R).$

\begin{Definition}{\em 
Let $R=S/I$ be a standard graded $K$--algebra. The algebra $R$ is called {\em Gorenstein} if it is Cohen-Macaulay of type $1.$ 
}
\end{Definition}
Hence, a Cohen-Macaulay standard graded $K$--algebra $R$ is Gorenstein if and only if $\omega_R\cong R(a)$ for some integer $a.$
The minimal free resolution of a Gorenstein algebra $R=S/I$ is self-dual. In particular, $\beta_i^S(R)=\beta_{n-d-i}^S(R).$

\begin{Definition}{\em
Let $R=S/I$ be a standard graded Cohen-Macaulay $K$--algebra. The number
\[
a(R)=-\min\{i: (\omega_R)_i\neq 0\} 
\] is called the {\em $a$--invariant} of $R.$
}
\end{Definition}

\begin{Proposition}\cite[Corollary 3.6.14]{BHbook}\label{cor3.6.14}
Let $R=S/I$ be a standard graded Cohen-Macaulay $K$--algebra with canonical module $\omega_R$ and $\yb=y_1,\ldots,y_m$ an $R$--sequence
of homogeneous elements with $\deg y_i=a_i$ for $1\leq i\leq m.$ Then 
\[
\omega_{R/\yb R}\cong (\omega_R/\yb \omega_R)(\sum_{i=1}^m a_i).
\] In particular, 
$
a(R/\yb R)=a(R)+\sum_{i=1}^m a_i.
$
\end{Proposition}

\subsection{Hibi rings and ideals}
\label{subsectHibi}

In this subsection we describe a class of rings and binomial ideals which were introduced by Hibi in \cite{Hibi}. They are associated with finite distributive lattices. 

Let $L$ be a distributive lattice and $P=\{p_1,\ldots,p_n\}$ its set of join-irreducible elements. Thus, $L=\MI(P).$ Let $K$ be a field 
and $R=K[t,x_1,\ldots,x_n]$ be a polynomial ring in $n+1$ indeterminates. Let $R[L]$ be the subring of $R$ which is generated over $K$ by 
the set of monomials $\{t\prod_{p_i\in \alpha}x_i: \alpha\in \MI(P)\}.$  Hibi showed in \cite{Hibi} that $R[L]$ is an algebra with straightening laws (ASL in brief) on $P$. We recall here the definition of an ASL. The reader may consult \cite{Eis} for a quick introduction to this topic. 

Let $A$ be a $K$--algebra, $H$ a finite poset, and $\varphi:H\to A$ an injective map. We identify $x\in H$ with $\varphi(x)\in A.$ A {\em standard monomial} in $A$ is a monomial of the form $\alpha_1\ldots \alpha_k$ where $\alpha_1\leq \cdots \leq \alpha_k.$

\begin{Definition}{\em 
The algebra $A$ is called an {ASL} on $H$ over $K$ if the following hold:
\begin{itemize}
	\item [(ASL-1)] The set of standard monomials form a $K$--basis of $A;$
	\item [(ASL-2)] If $\alpha,\beta\in H$ are incomparable and if $\alpha\beta=\sum r_i \gamma_{i1}\ldots \gamma_{ik_i},$ where $r_i\in K\setminus\{0\}$ and $\gamma_{i1}\leq \ldots \leq \gamma_{ik_i},$ is the unique expression of $\alpha\beta$ as a linear combination of standard monomials, then $\gamma_{i1}\leq \alpha, \beta$ for all $i.$
\end{itemize}
}
\end{Definition}

The relations in axiom (ASL-$2$) are called the {\em straightening relations} of $A.$ They generate the presentation ideal of $A;$ see 
\cite[Theorem 3.4]{Eis}.

\medskip

Let $L=\MI(P)$ be a distributive lattice with $P=\{p_1,\ldots,p_n\}$ and $$\varphi:L\to R=K[t,x_1,\ldots, x_n]$$ given by 
\[
\varphi(\alpha)=t\prod_{p_i\in \alpha} x_i \text{ for  } \alpha\in L.
\]

One observes that, for any $\alpha,\beta\in L, $
\begin{equation}\label{eq1}
\varphi(\alpha)\varphi(\beta)=\varphi(\alpha\cap\beta)\varphi(\alpha\cup\beta).
\end{equation}

We  show now that the Hibi ring $R[L]$ is an ASL on $L$ over $K.$
Axiom (ASL-2) is a straightforward consequence of equality (\ref{eq1}).

For (ASL-1), it is enough to show that for any two chains $\alpha_1\subseteq \cdots \subseteq \alpha_k$ and $\beta_1\subseteq \cdots \subseteq \beta_\ell$ in $L,$ we have $\varphi(\alpha_1)\ldots \varphi(\alpha_k)=\varphi(\beta_1)\ldots \varphi(\beta_\ell)$ if and only if  $k=\ell$ and $\alpha_i=\beta_i$ for all $i.$ This will imply that the standard monomials are distinct, so they form a $K$--basis of $R.$

Let $\varphi(\alpha_1)\cdots \varphi(\alpha_k)=\varphi(\beta_1)\cdots \varphi(\beta_\ell)$, that is, 
$t^k\prod_{i=1}^k (\prod_{p_j\in \alpha_i}x_j)=t^\ell\prod_{i=1}^\ell (\prod_{p_j\in \beta_i}x_j).$ This equality obviously implies that 
$k=\ell$ and $\prod_{i=1}^k (\prod_{p_j\in \alpha_i}x_j)=\prod_{i=1}^k (\prod_{p_j\in \beta_i}x_j).$ Therefore, we have 
\[(\prod_{p_i\in \alpha_1}x_i)^k (\prod_{p_i\in \alpha_2\setminus \alpha_1}x_i)^{k-1}\cdots (\prod_{p_i\in \alpha_k\setminus \alpha_{k-1}}x_i)=(\prod_{p_i\in \beta_1}x_i)^k (\prod_{p_i\in \beta_2\setminus \beta_1}x_i)^{k-1}\cdots (\prod_{p_i\in \beta_k\setminus \beta_{k-1}}x_i).
\] This equality implies that $\alpha_i=\beta_i$ for $1\leq i\leq k.$

 Thus, we have proved that $R[L]$ is an ASL. Since the 
straightening relations generate the presentation ideal of $R[L],$ we get $R[L]\cong K[L]/I_L$ where $K[L]=K[\{x_\alpha: \alpha\in L\}]$
 and $I_L=(x_\alpha x_\beta-x_{\alpha\cap\beta}x_{\alpha\cup\beta}: \alpha,\beta\in L, \alpha,\beta \text{ incomparable}).$ 

The presentation ideal $I_L$ is called the (binomial\footnote{One may find also the notion of monomial Hibi ideal in  literature. But we do not discuss this notion in this paper. Therefore, we will omit "binomial" when we refer to binomial Hibi ideals.}) {\em Hibi ideal} or the {\em join-meet ideal} of $L.$

\subsubsection{Gr\"obner bases of Hibi ideals}\label{GBsubsect} As above, let $K[L]$ be the polynomial ring in the variables $x_\alpha$ with $\alpha$ in 
$L$ and $I_L\subset K[L]$ the Hibi ideal associated with $L$. We order linearly  the  variables of $K[L]$ such that $x_\alpha\leq x_\beta$ if $\alpha\subseteq\beta$ in $L.$ We consider the reverse lexicographic order $<$ on $K[L]$ induced by this order of the variables. 

The following theorem appears in \cite[Chapter 10]{HH10}. We give here a different proof.

\begin{Theorem} \cite[Theorem 10.1.3]{HH10}\label{GBHibi}
The generators of $I_L$ form the reduced Gr\"obner basis of $I_L$ with respect to $<.$
\end{Theorem}

\begin{proof} For $\alpha,\beta\in L,$ we set $f_{\alpha,\beta}=x_\alpha x_\beta-x_{\alpha\cap\beta} x_{\alpha\cup\beta}.$ Obviously, 
$f_{\alpha,\beta}=0$ if and only if $\alpha$ and $\beta$ are comparable. If $\alpha,\beta$ are incomparable, then $\ini_<(f_{\alpha,\beta})=x_\alpha x_\beta.$

According to Buchberger's criterion, it is enough to show that  all the $S$--polynomials $S_<(f_{\alpha,\beta},f_{\gamma,\delta})$ reduce to zero for any pair of generators $f_{\alpha,\beta},f_{\gamma,\delta}$ of $I_L.$ If $\ini_<(f_{\alpha,\beta})$ and $\ini_<(f_{\gamma,\delta})$ are relatively prime, then it is known that $S_<(f_{\alpha,\beta},f_{\gamma,\delta})$ reduces to $0;$ see 
\cite[Poposition 2.15]{EH}. It remains to show that any $S$--polynomial of the form $S_<(f_{\alpha,\beta},f_{\alpha,\gamma})$ reduces to $0.$ But this follows immediately since one may easily check that the following equality is a standard expression of $S_<(f_{\alpha,\beta},f_{\alpha,\gamma})$:
\[
S_<(f_{\alpha,\beta},f_{\alpha,\gamma})=x_{\alpha\cup\gamma}f_{\beta,\alpha\cap\gamma}-x_{\alpha\cup\beta}f_{\gamma,\alpha\cap\beta}
+x_{\alpha\cap\beta\cap\gamma}(f_{\alpha\cup\gamma,\beta\cup(\alpha\cap\gamma)}-f_{\alpha\cup\beta,\gamma\cup(\alpha\cap\beta)}).
\]
\end{proof}

The above theorem has important consequences for the Hibi ring $R[L].$ In the first place, by Theorem~\ref{GBHibi}, it follows that 
\[
\ini_<(I_L)=(x_\alpha x_\beta: \alpha,\beta\in L, \alpha, \beta \text{ incomparable}).
\]
Hence, $\ini_<(I_L)$ is a squarefree monomial ideal. A well-known theorem of Sturmfels \cite[Theorem 5.16]{EH} implies that $R[L]$ is a normal domain. A theorem due to Hochster \cite{Ho72}, shows that $R[L]$ is Cohen-Macaulay. 

The Cohen-Macaulay property of $I_L$ may be deduced also in the following way. The initial ideal $\ini_<(I_L)$ is the Stanley-Reisner ideal \footnote{For the Stanly-Reisner theory we refer the reader to the monographs \cite{BHbook, Sta2}.} of the order complex of $L. $ This is defined as the complex of all the chains in $L.$ By a theorem of Bj\"orner \cite[Theorem 5.1.12]{BHbook}, this complex is shellable, thus its Stanley-Reisner ideal is Cohen-Macaulay. Then $I_L$ is Cohen-Macaulay as well; see \cite[Corollary 3.3.5]{HH10}.

\subsubsection{Some comments} We end this section with a few comments. 

One may obviously consider the following more general settings. Let $L$ be an arbitrary lattice, hence not necessarily distributive, and 
$K[L]$ the polynomial ring $K[\{x_\alpha: \alpha\in L\}]$. Let $I_L=(f_{\alpha,\beta}: \alpha,\beta\in L, \alpha,\beta \text{ incomparable}).$ Hibi showed in \cite{Hibi} and it is easily seen that $I_L$ is a prime ideal if and only if $L$ is distributive.

One may naturally ask whether $I_L$ is however a radical ideal when $L$ is not distributive. This is not the case and one may check, for instance, that for the lattice $L$ given in the left side of Figure~\ref{nondis}, $I_L$ is not a radical ideal. On the other hand, if $L$ is a pentagon (Figure~\ref{pentagon}), then $I_L$ is radical. The following problem would be of interest.

\begin{Problem}\label{Pb1}
Find classes of non-distributive lattices $L$ with the property that $I_L$ is a radical ideal.
\end{Problem}

If $I_L$ is a radical ideal, then its minimal prime ideals may be described in terms of the combinatorics of $L;$ see \cite[Section 2]{EHi}.

A reverse lexicographic order $<$ in $K[L]$ with the property that $\rank \alpha<\rank\beta$ implies that $\alpha<\beta$ is called a {\em 
rank reverse lexicographic order}. The following theorem from \cite{HHNach} characterizes the distributive lattices amongst the modular lattices in terms of the Gr\"obner bases of their ideals. 

\begin{Theorem}\cite[Theorem 2.1]{HHNach}
Let $L$ be a modular lattice. Then $L$ is distributive if and only if $I_L$ has a squarefree Gr\"obner basis with respect to any rank reverse lexicographic order.
\end{Theorem}

Moreover, in the same paper, the authors conjectured that if $L$ is modular, then, for any monomial order, $\ini_<(I_L)$ is not squarefree, unless $L$ is distributive. This conjecture was proved in \cite{EHi}.

\begin{Theorem}\cite{EHi}\label{HH conjecture}
Let $L$ be  a modular non-distributive lattice. Then, for any monomial order $<$ on $K[L]$, the initial ideal $\ini_<(I_L)$ is not squarefree.
\end{Theorem}

Note that, for the diamond lattice $L$ (the lattice displayed in Figure~\ref{nondis} in the right side), $I_L$ is radical. However,
as $L$ is modular, none of its initial ideals $\ini_<(I_L)$ is squarefree.  This simple example shows that the approach of  Problem~\ref{Pb1} is not so easy. One of the most common techniques to show that a polynomial ideal $I$ is radical is to find an initial ideal of $I$ which is radical. Unfortunately, as we have seen in Theorem~\ref{HH conjecture}, this technique cannot be applied in approaching Problem~\ref{Pb1}.
 
\subsection{The canonical module of a Hibi ring}
\label{canonicalHibi}

Let $L=\MI(P)$ be a distributive lattice  with $P=\{p_1,\ldots,p_n\}$ and $R[L]\subset K[t,x_1,\ldots,x_n]$ the associated Hibi ring. 
As we have already seen,  $R[L]$ is and ASL on $L$ over $K$ which  has as $K$--basis the standard monomials. This implies that every monomial in $R[L]$ is of the form $t^{w_0}x_1^{w_1}\cdots x_n^{w_n}$ where $(w_0,w_1,\ldots,w_n)\in \NN^{n+1}$ with $w_0\geq w_
i$ for all $i$ and $w_i\geq w_j$ if $p_i\leq p_j$ in $P.$ 

Since $R[L]$ is a domain, the canonical module $\omega_L$ of $R[L]$ is an ideal of $R[L];$ see \cite[Proposition 3.3.18]{BHbook}.
By a theorem of Stanley \cite{Sta3}, a $K$--basis of the canonical ideal $\omega_L$ is given by the monomials $t^{w_0}x_1^{w_1}\cdots x_n^{w_n}\in R[L]$ with $w_0>w_i>0$ for all $i$ and $w_i>w_j$ if $p_i<p_j$ in $P.$ 

Let $\hat{P}=P\cup\{-\infty,\infty\}$ be the poset defined in Subsection~\ref{combinat} and $\MS(P)$ the set of all functions 
$v:\hat{P}\to \NN$ with $v(\infty)=0$ and $v(p)\leq v(q)$ if $p\geq q$ in $\hat{P}.$ A function $v$ as above is called an 
{\em order reversing map}. A function $v\in \MS(P)$ is  a {\em strictly order reversing map} if $v(p)<v(q)$ if $p>q$ in $\hat{P}.$
Let $\MT(P)$ be the set of all strictly order reversing maps on $\hat{P}.$ Then, from what we said above, it follows that a $K$--basis of the canonical ideal $\omega_L$ is given by the set  $\{v^{(-\infty)}\prod_{i=1}^n x_i^{v(p_i)}: v\in \MT(P)\}.$

On $\MT(P)$ one defines the following partial order. For $v,v^\prime\in \MT(P),$ we set $v\geq v^\prime$ if the following conditions hold:
\begin{itemize}
	\item [(i)] $v(p)\geq v^\prime(p)$ for all $p\in \hat{P},$
	\item [(ii)] the function $v-v^\prime\in \MS(P)$, where $v-v^\prime: \hat{P}\to \NN$ is defined by $(v-v^\prime)(p)=v(p)-v^\prime(p)$ for 
	all $p\in \hat{P}.$
\end{itemize}

It follows that the minimal generators of $\omega_L$ are in one-to-one correspondence with the minimal elements of the poset $\MT(P).$
In particular, $R[L]$ is Gorenstein if and only if $\MT(P)$ has a unique minimal element. 

The following theorem was proved in \cite[\S 3]{Hibi}. Before stating it, we need to introduce some notation. For $x\in \hat{P},$ $\depth x$ denotes the rank of the subposet of $\hat{P}$ consisting of all elements $y\geq x$ in $\hat{P}$, and $\height x$ denotes the rank of the subposet of $\hat{P}$ which consists of all $y\in \hat{P}$ with $y\leq x.$ The number $\coheight x=\rank\hat{P}-\height x$ is called the {\em coheight} of $x.$ It is clear that the functions $\depth$ and $\coheight$ belong to $\MT(P).$ In addition, on easily sees that, for any $x,y\in \hat{P}$ with  $x\gtrdot y,$ we have $\depth y\geq \depth x+1$ and $\height x \geq \height y+1.$ If $P$ is pure, then 
$\depth x+ \height x =\rank \hat{P}$ for any $x\in \hat{P}.$

Let $v\in \MT(P)$ and $-\infty<p_0<p_1<\cdots < p_r<\infty$ be a maximal chain in $\hat{P}$ 
with $r=\rank P.$ Then $v(-\infty)>v(p_0)>v(p_1)>\cdots > v(P_r)>v(\infty)=0$, which implies that 
\begin{equation}\label{firstineq}
v(-\infty)\geq \rank \hat{P}=\rank P+2.
\end{equation}

With similar arguments, one shows that, for all $x\in \hat{P},$
\begin{equation}\label{secondineq}
v(x)\geq \depth x.
\end{equation}

\begin{Theorem}\cite{Hibi}\label{GorensteinHibi}
The Hibi ring $R[L]$ is Gorenstein if and only if $P$ is pure.
\end{Theorem}

\begin{proof}
To begin with, let $P$ be pure and $x,y\in \hat{P}$ with $x\gtrdot y.$ We get 
\[
\height y+\depth x +1=\rank\hat{P}=\height y+\depth y.
\] This implies that $\depth y=\depth x+1.$ By using this equality, we show that $\depth$ is the unique minimal element of $\MT(P).$ 
Let $v\in \MT(P).$ Obviously, $v(x)\geq \depth x$ for all $x\in \hat{P}.$ Let now $x,y\in \hat{P}$ with $x\gtrdot y.$ Then 
$v(x)-\depth x\leq v(y)-\depth y$ since $v(x)-v(y)\leq -1=\depth x-\depth y.$ Clearly, the inequality $v(x)-\depth x\leq v(y)-\depth y$
extends to any $x>y$ in $\hat{P}$ which shows that $v\geq \depth $ in $\MT(P).$

Conversely, let $R[L]$ be a Gorenstein  ring, that is, $\MT(P)$ has a unique minimal element. Assume that $P$ is not pure. Then there must be $x,y\in \hat{P}$ with $x\gtrdot y$ such that $\depth y>\depth x+1.$ We define 
$w\in \MT(P)$ as follows,
\[
w(z)=\left\{
\begin{array}{ll}
	1+\depth z, & z\leq x, z\neq y,\\
	\depth z, & \text{ otherwise.}
\end{array}
\right.\]
Then $w(z)\geq \depth z$ for all $z$ and $w(x)-\depth x=1> w(y)-\depth y=0.$ This shows that $w$ and $\depth$ are incomparable in 
$\MT(P)$ which implies that $\MT(P)$ has at least two minimal elements, a contradiction to our hypothesis.
\end{proof}

\begin{Examples}{\em
\begin{itemize}
	\item [1.] For the lattice $L$ displayed in Figure~\ref{Birhoff_fig}, the ring $R[L]$ is Gorenstein since the poset of the join-irreducible elements is pure.
	\item [2.] Let $P=\{p_1,p_2,p_3\}$ with $p_1<p_2.$ This poset is not pure, thus the Hibi ring of the lattice $L=\MI(P)$ is not Gorenstein.
\end{itemize}

}
\end{Examples}

\subsection{Generalized Hibi rings}
\label{generalized}

Hibi rings were generalized in \cite{EHM}. Let $P=\{p_1,\ldots,p_n\}$ be a poset and $\MI(P)$ the ideal lattice of $P.$ We fix a positive 
integer $r.$ An {\em $r$--multichain} in $P$ is a chain of poset ideals of $P$ of length $r$:
\[
\MI: \hspace{0.7cm} I_1\subseteq I_2\subseteq\cdots\subseteq I_{r-1}\subseteq I_r=P.
\]

Let $\MI_r(P)$ be the set of all $r$--multichains in $P.$ If $\MI: I_1\subseteq\cdots\subseteq I_{r-1}\subseteq I_r=P$ and 
$\MJ: J_1\subseteq\cdots\subseteq J_{r-1}\subseteq J_r=P$ are two $r$--multichains in $P,$ then
\[
\MI\cup \MJ: I_1\cup J_1\subseteq\cdots\subseteq I_{r-1}\cup  J_{r-1}\subseteq I_r\cup J_r=P
\]
and
\[
\MI\cap \MJ: I_1\cap J_1\subseteq\cdots\subseteq I_{r-1}\cap  J_{r-1}\subseteq I_r\cap J_r=P
\]
belong to $\MI_r(P)$ as well, hence $\MI_r(P)$ is a distributive lattice. 

With each $r$--multichain $\MI$ in  $\MI_r(P)$ we associate a monomial $u_{\MI}$ in the polynomial ring 
$S=K[\{x_{ij}:1\leq i\leq r,1\leq j\leq n\}]$ which is defined as
\[
u_{\MI}=x_{1J_1}x_{2J_2}\cdots x_{rJ_r}
\]
where 
\[
x_{kJ_k}=\prod_{p_\ell\in J_k}x_{k\ell} \text{ and } J_k=I_k\setminus I_{k-1} \text{ for all } 1\leq k\leq r.
\]

Let $R_r(P)$ be the $K$--subalgebra of $S$ generated by the set $\{u_{\MI}:\MI\in \MI_r(P)\}$. The ring $R_r(P)$ is called a 
{\em generalized Hibi ring.}

For example, for $r=2,$ an $r$--multichain of $P$ is of the form $I\subseteq P$ where $I$ is a poset ideal of $P.$ If we set $x_{1j}=x_j$
 and $x_{2j}=y_j$ for $1\leq j\leq n,$ then to the multichain $I\subset P$ we associate the monomial 
$\prod_{p_i\in I}x_i \prod_{p_i\notin I}y_i$. Hence, \[R_2(P)=K[\{\prod_{p_i\in I}x_i \prod_{p_i\notin I}y_i: I\in \MI(P)\}].\]
The ring $R_2(P)$ is isomorphic to the classical Hibi ring associated with the lattice $L=\MI(P)$ since they have the same defining relations as it follows as a particular case of Corollary~\ref{ASLcor}.

Similarly to the classical Hibi rings, we get the following result.

\begin{Theorem}\label{ASL general}
The ring $R_r(P)$ is an ASL on $\MI_r(P)$ over $K.$
\end{Theorem}

\begin{proof}
The proof is similar to the corresponding statement for Hibi rings. 

Let $\psi: \MI_r(P)\to S$ defined by $\psi(\MI)=u_{\MI}$ for all $\MI\in \MI_r(P).$ One may check that 
\[
\psi(\MI)\psi(\MJ)=\psi(\MI\cap\MJ)\psi(\MI\cup\MJ)
\] for all $\MI,\MJ\in \MI_r(P);$ see also \cite[Lemma 2.1]{EHM}. This equality shows that $R_r(P)$ satisfies axiom (ASL-2). For showing 
(ASL-1), one may proceed as in Subsection~\ref{subsectHibi} and show that the standard monomials in $R_r(P)$ are distinct.  Indeed, let $\MI_1\subseteq\cdots \subseteq\MI_t$ and 
$\MJ_1\subseteq \cdots \subseteq \MJ_s$ be two chains in $\MI_r(P)$ such that 
\[
\psi(\MI_1)\cdots \psi(\MI_t)=\psi(\MJ_1)\cdots \psi(\MJ_s)
\]
which is equivalent to 
\[
\prod_{q=1}^t u_{\MI_q}=\prod_{q=1}^s u_{\MJ_q}
\]
or, more explicitly, 
\[
\prod_{q=1}^t(\prod_{k=1}^r x_{k,\MI_{q,k}\setminus \MI_{q, k-1}})=\prod_{q=1}^s(\prod_{k=1}^r x_{k,\MJ_{q,k}\setminus \MJ_{q, k-1}}).
\] From this last equality it follows 
\[
\prod_{q=1}^t(\prod_{k=1}^\ell x_{k,\MI_{q,k}\setminus \MI_{q, k-1}})=\prod_{q=1}^s(\prod_{k=1}^\ell x_{k,\MJ_{q,k}\setminus \MJ_{q, k-1}})
\] 
for all $\ell\in \{1,\ldots,r\}.$
Taking $\ell=1$ in the above equality we derive $t=s$ and $I_{q1}=J_{q1}$. Next, by inspecting the above equalities step by step for $\ell=2,\ldots,r$, we get 
$\MI_{q,k}=\MJ_{q,k}$ for all $q$ and $k.$
\end{proof}

Let $T$ be the polynomial ring in the variables $y_{\MI}$ with $\MI\in \MI_r(P)$ and $\varphi: T\to R_r(P)$ the $K$--algebra 
homomorphism induced by $y_{\MI}\mapsto u_{\MI}$ for all $\MI\in \MI_r(P)$.  Theorem~\ref{ASL general} has the following consequence.

\begin{Corollary}\label{ASLcor}
The presentation ideal of the ring $R_r(P)$ is generated by the binomials $y_{\MI}y_{\MJ}-y_{\MI\cap\MJ}y_{\MI\cup\MJ}$ where $
\MI,\MJ\in \MI_r(P)$ are incomparble $r$--multichains.
\end{Corollary}

We fix a linear order on the variables $y_{\MI}$
such that $y_{\MI}<y_{\MJ}$ if $\MI\subset \MJ.$ Corollary~\ref{ASLcor} shows that $R_r(P)$ is the classical Hibi ring of $\MI_r(P)$, thus we get the following statement.

\begin{Theorem}\cite[Theorem 4.1]{EHM}\label{GBgeneral}
The set
\[
\MG=\{y_{\MI}y_{\MI^\prime}-y_{\MI\union \MI^\prime}y_{\MI\sect \MI^\prime}\in T : \MI,\MI^\prime\in \MI_r(P) \text{ incomparable}\},
\]
is  the  reduced Gr\"obner basis of the ideal $\Ker\varphi$ with respect to the reverse lexicographic order induced by the given order of the
variables $y_{\MI}$.
\end{Theorem}

\begin{Corollary}\cite[Corollary 4.2]{EHM}
The ring $R_r(P)$ is a Cohen-Macaulay normal domain.
\end{Corollary}

In order to have a better knowledge of $R_r(P),$ we need to identify the join-irreducible elements of $\MI_r(P).$
Let  $Q_{r-1}$ denote the set $[r-1]=\{1,\ldots,r-1\}$
 endowed with the natural order.

\begin{Theorem}\cite[Theorem 4.3]{EHM}\label{thm43}
Let $P$ be a finite poset. Then, for any  $r\geq 2,$ $R_r(P)\cong R_2(P\times Q_{r-1}).$ 
\end{Theorem}

\begin{proof}
We have to show that the poset $P^\prime$ of the join-irreducible elements of $\MI_r(P)$ is isomorphic to $P\times Q_{r-1}.$ 

In the first place we identify the join-irreducible elements of $P^\prime.$  Let
\[
\MI: \emptyset\subseteq\emptyset\subseteq\cdots\subseteq\emptyset \subset I_k\subseteq \cdots \subseteq I_r=P
\] be an $r$--multichain of $\MI_r(P).$ We claim that $\MI$ is join-irreducible if and only if $I_k$ is a join-irreducible poset ideal in 
$P$ and $I_k=I_{k+1}=\cdots= I_{r-1}.$ The if part is obvious. For the only if part, let us first assume that $I_k=J\cup J^\prime$ with 
$J,J^\prime$ poset ideals different from $I_k.$ Then, we may decompose $\MI=\MJ\cup \MJ^\prime$ where
\[
\MJ: \emptyset\subseteq\emptyset\subseteq\cdots\subseteq\emptyset \subset J\subseteq \cdots \subseteq I_r=P
\]
and
\[
\MJ^\prime: \emptyset\subseteq\emptyset\subseteq\cdots\subseteq\emptyset \subset J^\prime\subseteq \cdots \subseteq I_r=P,
\] a contradiction.

Suppose now that there exists an integer $s$ with $k\leq s<r-1$ such that $I_k=I_{k+1}=\cdots=I_s$   and $I_s\subset I_{s+1}$. Then $\MI= \MJ\union \MJ'$ where
\[
\MJ\: \emptyset \subseteq \cdots \subseteq \emptyset \subset I_k\subseteq I_{k}\subseteq \cdots \subseteq I_k\subseteq I_r
\]
and
\[
\MJ'\: \emptyset \subseteq \cdots \subseteq \emptyset \subset I_s \subseteq I_{s+1}\subseteq \cdots \subseteq I_r,
\]
a contradiction.

Let $\MI: \emptyset\subseteq \cdots\subseteq \emptyset \subset I=I=\cdots=I\subset P$ with $k$ copies of $I$ where $I$ is a join irreducible element of $\MI(P).$ Then $I$ is a principal ideal in $\MI(P),$ hence there exists a unique element $p\in I$ such that $I=\{a\in P : a\leq p\}.$
We define the  poset isomorphism between the poset $P^\prime$ of the join irreducible elements of $\MI_r(P)$ and $P\times Q_{r-1}$ as follows. To $\MI: \emptyset\subseteq \cdots \subseteq\emptyset \subset I=I=\cdots=I\subset P$ with $k$ copies of $I$ we assign $(p,k)\in P\times Q_{r-1}$.
\end{proof}

The above theorem allows us to extend Theorem~\ref{GorensteinHibi} to generalized Hibi rings. 

\begin{Corollary}\cite[Corollary 4.5]{EHM}\label{corgengor}
Let $r\geq 2$ be an integer. The ring $R_r(P)$ is Gorenstein if and only if $P$ is pure.
\end{Corollary}

\begin{proof}
By Theorem~\ref{thm43} and Theorem~\ref{GorensteinHibi}, we only need to observe that the poset $P^\prime$ of the join irreducible elements of $\MI_r(P)$ is pure if and only if $P$ is pure.
\end{proof}

\section{Level and pseudo-Gorenstein Hibi rings}
\label{pseudo}

In Theorem~\ref{GorensteinHibi} we presented the characterization of Gorenstein Hibi rings in terms of the poset $P$ of the join-irreducible elements of the lattice. In this section we study two weaker properties of $R[L].$ More precisely, we will characterize the Hibi rings which are  pseudo-Gorenstein and give necessary and sufficient conditions for levelness. This section is mainly based on  paper \cite{EHHSara}.

\subsection{Level and pseudo-Gorenstein algebras}
\label{pseudo-algebras}

Let $K$ be a field and $R$ a standard graded $K$--algebra. We assume that $R$ has the presentation $R=S/I$ where $S=K[x_1,\ldots,x_n]$ is a polynomial ring over $K$ and $I\subset S$ is a graded ideal. We also make the assumption that $R$ is Cohen-Macaulay of dimension $d.$ Let $\omega_R$ denote the canonical module of $R$ and $a=\min\{i: (\omega_R)_i\neq 0\}$.

As we have already seen in Subsection~\ref{subsectcanonical}, $R$ is Gorenstein if and only if $\omega_R$ is a cyclic $R$--module. Let
\[
\FF:\hspace{0.5cm} 0\to F_{n-d}\to \cdots \to F_1\to F_0\to R\to 0
\]
be the minimal graded free resolution of $R$ over $S.$

The notion of level rings was introduced in \cite{Sta4}.

\begin{Definition}\label{leveldefn}{\em 
The algebra $R$ is called {\em level} if all the generators of the canonical module $\omega_R$ have the same degree.}
\end{Definition}

In other words, $R$ is level if and only if the generators of $F_{n-d}$ are of  same degree. 

The following notion was introduced in \cite{EHHSara}.

\begin{Definition}\label{pseudodefn}{\em
The algebra $R$ is called {\em pseudo-Gorenstein} if $\dim_K(\omega_R)_a=1$.}
\end{Definition}

It is already clear from the above definitions that an algebra $R$ is Gorenstein if it is  level and pseudo-Gorenstein.

On the other hand, we may easily prove the following characterization of pseudo-Gorensteiness.

\begin{Proposition}\label{echivalent}
Let $R$ be a Cohen-Macaulay standard graded $K$--algebra of $\dim R=d$ and canonical module $\omega_R.$ The following statements are equivalent:
\begin{itemize}
	\item [(i)] The algebra $R$ is pseudo-Gorenstein;
	\item [(ii)] Let $\yb=y_1,\ldots,y_d$ be an $R$--regular sequence of linear forms and $\bar{R}=R/\yb R.$ Let $b=\max\{i: (\bar{R})_i\neq 0\}$. Then $\dim_K (\bar{R})_b=1.$
	\item [(iii)] Let $H_R(t)=h(t)/(1-t)^d$ be the Hilbert series of $R.$ Then, the leading coefficient of $h$ is equal to $1.$
	\item [(iv)] The highest shift $c$ in the resolution $\FF$ of $R$ over $S$ appears in $F_{n-d}$ and $\beta_{n-d,c}(R)=1.$
 \end{itemize}
\end{Proposition}

\begin{proof}
We briefly sketch the main steps of the proof. Implication (iv)$\Rightarrow$(iii) follows immediately if we apply the additivity property of the Hilbert series to the resolution of $R.$ We get 
\[
H_R(t)=\frac{\sum_{i=0}^{n-d}(-1)^i \sum_j\beta_{ij}t^j}{(1-t)^{n}}.
\]
The leading coefficient of the numerator of $H_R(t)$ is equal to $(-1)^{n-d}$, hence, after simplifying the expression of $H_R(t)$ by $(1-t)^{n-d},$ we get the leading coefficient of $h(t)$ equal to $1.$

For (iii)$\Rightarrow$(ii), we notice that $H_{R}(t)=H_{\bar{R}}(t)/(1-t)^d,$ thus $$H_{\bar{R}}(t)=\sum_{i=0}^b\dim_K(\bar{R})_i t^i=h(t).$$
This equality leads to the desired conclusion.

Implication (ii)$\Rightarrow$(i) follows by Proposition~\ref{cor3.6.14} combined with the fact that the canonical module of $\bar{R}$ is 
$\Hom_K(\bar{R},K)$ since $\bar{R}$ is Artinian; see \cite[Theorem 3.3.7]{BHbook}.

Finally, (i)$\Rightarrow$(ii) is obvious since the resolution of $\omega_R$ is the dual of $\FF.$
\end{proof}

In the next two subsections, we will study level and pseudo-Gorenstein Hibi rings. It will turn out that  the property of $R[L]$ of being pseudo--Gorenstein or level does not depend on the field. Therefore,  we may also say that $L$ is pseudo--Gorenstein or level if the Hibi ring is so.

\subsection{Pseudo-Gorenstein Hibi rings}
\label{pseudoHibi}

Let $L$ be a distributive lattice and $P$ the subposet of its join-irreducible elements. Let $S=K[\{x_\alpha:\alpha\in L\}]$ and $I_L$ the Hibi binomial ideal associated with $L.$ As we have seen in Subsection~\ref{canonicalHibi}, the canonical ideal $\omega_L$ of  $R[L]=S/I_L$ has the minimal generators in one-to-one correspondence with the minimal elements of the poset $\MT(P)$ which consists of all 
strictly order reversing maps $v:\hat{P}\to \NN$ with $v(\infty)=0.$ It then follows that $R[L]$ is pseudo-Gorenstein if and only if $\MT(P)$ contains exactly one minimal element.

\begin{Theorem}\cite[Theorem 2.1]{EHHSara}\cite[Corollary 3.15.18 (a)]{Sta1}\label{classification}
The ring $R[L]$ is pseudo-Gorenstein if and only if, for all $x\in P,$ we have
$\depth(x)+\height(x)=\rank \hat{P}.$
\end{Theorem}

\begin{proof}
Let $R[L]$ be pseudo-Gorenstein. Then $\omega_L$ has a unique minimal generator of least degree which is actually $\rank \hat{P}.$ Since
the maps $\depth$ and $\coheight$  correspond to generators of degree equal to $\rank\hat{P}$, they must be equal. This leads to the desired equality.

Conversely, let us assume that for all $x\in P,$ we have $\depth(x)+\height(x)=\rank \hat{P}.$ This implies that, for any  $x\in \hat{P},$ 
there exists a chain $C$ of length equal to $\rank \hat{P}$ with $x\in C$. Let $v\in \MT(P)$ with  $v(-\infty)=\rank \hat{P}$. Then, 
for any $y\in C,$ we must have $v(y)=\depth(y)$. In particular, $v(x)=\depth(x)$. Hence, $v$ is uniquely determined which implies that $L$ 
is pseudo-Gorenstein.
\end{proof}

In Figure~\ref{exampleGor} we represent the posets $P$ for a pseudo-Gorenstein lattice which is not Gorenstein and a lattice which is not pseudo-Gorenstein.

\begin{figure}[hbt]
\begin{center}
\psset{unit=0.5cm}
\begin{pspicture}(-8,-1)(-4,4)
\rput(-12,0){
\rput(0,0){$\bullet$}
\rput(0,2){$\bullet$}
\rput(0,4){$\bullet$}
\rput(2,0){$\bullet$}
\rput(2,4){$\bullet$}
\rput(4,0){$\bullet$}
\rput(4,2){$\bullet$}
\rput(4,4){$\bullet$}
\psline(0,0)(0,4)
\psline(2,0)(2,4)
\psline(4,0)(4,4)
\psline(0,2)(2,0)
\psline(2,4)(4,2)
\rput(2,-1){Pseudo-Gorenstein}
}
\rput(0,0){
\rput(0,0){$\bullet$}
\rput(0,2){$\bullet$}
\rput(2,0){$\bullet$}
\psline(0,0)(0,2)
\rput(1,-1){Not Pseudo-Gorenstein}
}
\end{pspicture}
\end{center}
\caption{}\label{exampleGor}
\end{figure}

\subsection{Level Hibi rings}
\label{levelHibi}

The examples displayed in Figure~\ref{Hibiexample} are taken from \cite{Hibi2}. They show that it is not possible to decide the levelness of a Hibi ring only from its $h$--vector. More precisely, neither the number of components of the $h$--vector, nor its last component says anything about the level property of $R[L].$ 

\begin{figure}[hbt]
\begin{center}
\psset{unit=0.5cm}
\begin{pspicture}(-2,-2)(-4,6)
\rput(-12,0){
\pspolygon(0,0)(8,0)(8,2)(0,2)
\pspolygon(6,0)(8,0)(8,6)(6,6)
\pspolygon(4,0)(6,0)(6,4)(4,4)
\psline(2,0)(2,2)
\psline(6,4)(8,4)
\rput(4,-1.3){$h(R[L])=(1,7,9,2)$}
\rput(4,-2.5) {$R[L]$ is level}
}
\rput(-1,0){
\pspolygon(0,0)(6,0)(6,2)(0,2)
\pspolygon(4,2)(8,2)(8,4)(4,4)
\pspolygon(6,4)(8,4)(8,6)(6,6)
\psline(2,0)(2,2)
\psline(4,0)(4,2)
\psline(6,4)(6,2)
\rput(4,-1.3){$h(R[L])=(1,6,9,2)$}
\rput(4,-2.5) {$R[L]$ is not level}
}
\end{pspicture}
\end{center}
\caption{}\label{Hibiexample}
\end{figure}

We would like to make a short comment on how Figure~\ref{Hibiexample} should be interpreted. In order to be consistent with the previous pictures, we should have rotated the drawings counterclockwise with $45$ degrees. But usually, we use  representations of planar distributive lattices like in Figure~\ref{Hibiexample}  in order to  recognize easier the planar coordinates of the elements of the lattices.

\medskip

The first attempt to study the level property of a Hibi ring was done in \cite{M}. In that paper, a sufficient condition for levelness was given.

\begin{Theorem}\cite[Theorem 3.3]{M}\label{Myiazaki}
Let $L=\MI(P)$ be a distributive lattice. If the subposet $\{y\in P: y\geq x\}$ of $P$ is pure for all $x\in P,$ then $R[L]$ is level.
\end{Theorem}

\begin{proof}
We have to prove that all the minimal elements of $\MT(P)$ have the same degree, namely $\rank\hat{P}.$ Thus, it suffices to show that for any $v\in \MT(P)$ there exists $v_0\in \MT(P)$ with $v_0(-\infty)=\rank\hat{P}$ such that $v-v_0\in \MS(P),$ that is, $v\geq v_0$ 
in $\MT(P).$ 

Let $v\in \MT(P)$ and define $v_0:\hat{P}\to \NN$ by
\[
v_0(x)=\max\{\depth x, \rank\hat{P}-v(-\infty)+v(x)\}
\] for all $x\in P.$ Clearly, $v_0(-\infty)=\rank \hat{P}.$ We have to show that for any $y\gtrdot x$ in $\hat{P},$ we have 
$v_0(y)<v_0(x)$ and $v(y)-v_0(y)\leq v(x)-v_0(x).$ We first observe that our hypothesis implies that $\depth x=\depth y+1.$

If $v_0(y)=\rank\hat{P}-v(-\infty)+v(y),$ then $v_0(y)<\rank\hat{P}-v(-\infty)+v(x)\leq v_0(y).$ If $v_0(y)=\depth y,$ then 
$v_0(y)<\depth x\leq v_0(x).$

For the  second inequality, let us first take  $$v_0(x)=\depth x>\rank\hat{P}-v(-\infty)+v(x).$$ We get 
\[
\rank\hat{P}-v(-\infty)+v(y)<\rank\hat{P}-v(-\infty)+v(x)\leq \depth x-1=\depth y,
\] which implies that $v_0(y)=\depth y.$ Therefore, the inequality $v(y)-v_0(y)\leq v(x)-v_0(x)$ is equivalent to $$v(y)-v(x)\leq \depth y-\depth x=-1,$$ which is obviously true. Now, let $$v_0(x)=\rank\hat{P}-v(-\infty)+v(x).$$ It follows that 
\[v_0(y)-v_0(x)\geq (\rank\hat{P}-v(-\infty)+v(y))-(\rank\hat{P}-v(-\infty)+v(x))=v(y)-v(x)
\] which leads to the desired inequality.
\end{proof}

\begin{Remark}{\em 
By duality, one gets another sufficient condition for the levelness of the Hibi ring $R[L]:$ If  the subposet $\{y\in P: y\leq x\}$ of $P$ is pure for all $x\in P,$ then $R[L]$ is level.
}
\end{Remark}

In Figure~\ref{NotMy} is displayed a poset $P$ which  shows that neither the condition given in Theorem~\ref{Myiazaki} nor its dual is 
necessary for levelness. One may easily show that $R[\MI(P)]$ is level either directly, by computing the minimal elements of $\MT(P),$ or by using a computer to find the resolution of $R[\MI(P)]$. However, the poset does not satisfy any of the sufficient conditions of being level.

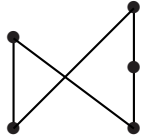
\begin{figure}[hbt]
\begin{center}
\psset{unit=0.8cm}
\begin{pspicture}(-5,2)(4,4)

\rput(-1.5,2){
\psline(0,0)(0,1.5)
\rput(0,0){$\bullet$}
\rput(0,1.5){$\bullet$}

\psline(2,0)(2,2)
\rput(2,0){$\bullet$}
\rput(2,1){$\bullet$}
\rput(2,2){$\bullet$}
\psline(0,0)(2,2)
\psline(0,1.5)(2,0)
}
\end{pspicture}
\end{center}
\caption{Butterfly poset}\label{NotMy}
\end{figure}

\begin{Remark}{\em 
In \cite{EHHSara}, the poset of Figure~\ref{NotMy} is called a {\em butterfly poset}. One may consult \cite{EHHSara} for more properties of butterfly posets.
}
\end{Remark}

A necessary condition for levelness was given in \cite{EHHSara}.

\begin{Theorem}\cite[Theorem 4.1]{EHHSara}
\label{alsoviviana}
Suppose   $L$ is level.  Then
\begin{eqnarray}\label{neccond}
\label{inequality}
\height(x)+\depth(y)  \leq \rank \hat{P}+1
\end{eqnarray}
for all $x,y\in P$ with  $x\gtrdot y$.
\end{Theorem}

\begin{proof}
Let $x,y\in P$ such that $x\gtrdot y$ and suppose that $\height(x)+\depth(y)  > \rank \hat{P}+1$. We have to show that $L$ is not level.

By our  assumption we get
\[
\height(x)+\depth(y)  > \rank \hat{P}+1\geq \height(x)+\depth(x)+1,
\]
and hence
\[
\depth(y)>\depth(x)+1.
\]
We show that there exists a minimal element $w\in \MT(P)$ with $w(-\infty)>\rank \hat{P}$. This then proves that $L$ is not level.

Let $\depth(y)-\depth(x)-1=\alpha.$ Then $\alpha>0$. We define $v\: \hat{P}\to \NN$ as follows:

\[
v(z)= \left\{ \begin{array}{ll}
       \depth(z)+\alpha, & \;\textnormal{if $x\geq z$, $z\neq y$}, \\ \depth(z), & \;\text{otherwise.}
        \end{array} \right.
\]
Then $v\in \MT(P)$. If $v\in \MT(P)$ is minimal, then we are done, since
\[
v(-\infty)= \depth(-\infty) +\alpha = \rank \hat{P} +\alpha\geq \rank \hat{P}+1.
\]
The last inequality follows from the fact that $\alpha>0$.

On the other hand, if $v$ is not  minimal  in $\MT(P)$, then   there exists a minimal element  $w\in\MT(P)$ with $v-w\in \MS(P)$. It follows that
\[
0\leq v(x)-w(x)\leq v(y)-w(y)=\depth(y)-w(y)\leq 0.
\]
Hence
\[
w(x)=v(x)= \depth(x)+\alpha= \depth(x)+\depth(y)-\depth(x)-1=\depth(y)-1.
\]
Let
\[
x=z_0>z_1>\cdots > z_k=-\infty
\]
be a chain whose length is $\height(x)$. Then
\[
w(x)<w(z_1)<\cdots < w(z_k)=w(-\infty),
\]
which implies that
\[
w(-\infty)\geq w(x)+\height(x)=(\depth(y)-1)+\height(x) > \rank \hat{P}.
\]
\end{proof}

In the next subsection, we will see that, for a class of planar lattices, condition (\ref{neccond}) is also sufficient for the level property of the Hibi ring.

\subsection{Regular hyper--planar lattices}
\label{subsecthyper}

Hyper-planar lattices generalize the planar lattices. They were introduced in \cite{EHHSara}.

\begin{Definition}\label{hyperdefn}{\em
Let $L$ be a finite distributive lattice and $P$ its poset of join-irreducible elements. The lattice  $L$ is called a  {\em hyper-planar lattice}, if $P$ as a set is the disjoint union of chains $C_1,\ldots, C_d$,  where each $C_i$ is a maximal chain in $P$.  We call such a chain decomposition {\em canonical}.
}
\end{Definition}

For $d=2,$ we recover simple planar lattices.

A canonical chain decomposition of the poset $P$ of join-irreducible elements for a hyper-planar lattice $L$ is, in general, not  uniquely determined. However, if $C_1\union C_2\union\cdots \union C_s$ and $D_1\union D_2\union \cdots \union D_t$ are canonical chain decompositions of $P$, then $s=t$.
Indeed, let  $\max(Q)$ denote the set of maximal elements of a finite poset $Q$. Then
\begin{eqnarray}
\label{max}
\max(P)&=&\max(C_1)\union \max(C_2)\union\cdots \union \max(C_s)\\
&=&\max(D_1)\union \max(D_2)\union\cdots \union \max(D_t).\nonumber
\end{eqnarray}
Let $\max(C_i)=\{x_i\}$ for $i=1,\ldots,s$ and $\max(D_i)=\{y_i\}$ for $i=1,\ldots,t$. Then the elements $x_i$ as well as the elements $y_i$ are pairwise distinct, and  it follows from (\ref{max}) that
\[
\{x_1,x_2,\ldots,x_s\}=\{y_1,y_2,\ldots,y_t\},
\]
Hence, t follows that $s=t$.

One would even  expect the equality
\begin{eqnarray}
\label{equal}
\{\ell(C_1),\ell(C_2),\ldots, \ell(C_s)\} =\{\ell(D_1),\ell(D_2),\ldots,\ell(D_t)\},
\end{eqnarray}
as multisets. However, this is not the case. The poset $P$ displayed in  Figure~\ref{different} has the following two canonical chain decompositions
\[
C_1=a<b<c<d<e<f, \quad C_2=g<h<i<j<k<l,
\]
and
\[
D_1=a<b<i<e<f, \quad D_2=g<h<c<d<j<k<l.
\]

We have  $\ell(C_1)=\ell(C_2)=5$, while  $\ell(D_1)=4$ and $\ell(D_2)=6$.

\begin{figure}[hbt]
\begin{center}
\psset{unit=0.8cm}
\begin{pspicture}(-5,0)(4,6)

\rput(-1.5,0){
\psline(0,0)(0,5)
\rput(0,0){$\bullet$}
\rput(0,1){$\bullet$}
\rput(0,2){$\bullet$}
\rput(0,3){$\bullet$}
\rput(0,4){$\bullet$}
\rput(0,5){$\bullet$}
\psline(2,0)(2,6)
\rput(2,0){$\bullet$}
\rput(2,1){$\bullet$}
\rput(2,2.5){$\bullet$}
\rput(2,4){$\bullet$}
\rput(2,5){$\bullet$}
\rput(2,6){$\bullet$}
\psline(0,1)(2,2.5)
\psline(0,4)(2,2.5)
\psline(0,3)(2,4)
\psline(0,2)(2,1)
\rput(-0.3,0){$a$}
\rput(-0.3,1){$b$}
\rput(-0.3,2){$c$}
\rput(-0.3,3){$d$}
\rput(-0.3,4){$e$}
\rput(-0.3,5){$f$}
\rput(2.3,0){$g$}
\rput(2.3,1){$h$}
\rput(2.3,2.5){$i$}
\rput(2.3,4){$j$}
\rput(2.3,5){$k$}
\rput(2.3,6){$l$}
}
\end{pspicture}
\end{center}
\caption{}\label{different}
\end{figure}

In order to guarantee that  equality (\ref{equal}) is  satisfied we have to add an extra condition on the hyper-planar lattice. 

\begin{Definition}\label{regulardefn}{\em 
The lattice $L=\MI(P)$ is called  {\em regular hyper-planar}, if, for any canonical chain decomposition $C_1\cup C_2\cup\ldots\cup C_d$ of $P$, and for all $x<y$ with $x\in C_i$ and $y\in C_j$ it follows that $\height_{C_i}(x)<\height_{C_j}(y)$.
}
\end{Definition}

In the next corollary we give some properties of regular hyper-planar lattices. First we need the following result. 

\begin{Lemma}\cite[Lemma 3.1]{EHHSara}
\label{regular}
Let $L$ be a regular hyper-planar lattice and $C_1\cup \ldots \cup C_d$  a canonical chain decomposition of $P$. Then, for all $i$ and $x\in C_i,$ we have $\height_{C_i}(x)=\height_P(x)$.
\end{Lemma}

\begin{proof}
We apply induction on $\height_P(x)$. If $\height_P(x)=0$, then there is nothing to show. Assume that $\height_P(x)>0$ and let $y\in P$ with $x\gtrdot y$ with $\height_P(y)=\height_P(x)-1$. Let us assume that  $y\in C_j$. Since $\height_P(y)=\height_P(x)-1,$ by the inductive hypothesis we  obtain
\[
\height_P(x)-1=\height_P(y)=\height_{C_j}(y)<\height_{C_i}(x)\leq \height_P(x).
\]
This yields the desired conclusion.
\end{proof}

\begin{Corollary}\cite[Corollary 3.2]{EHHSara}
\label{ell}
Let $L$ be a regular hyper-planar lattice with the distinct canonical chain decompositions $C_1\cup \ldots \cup C_d$ and $D_1\union D_2\union \cdots \union D_d$  of $P$. Then
\begin{enumerate}
\item[{\em (a)}] $\{\ell(C_1),\ell(C_2),\ldots, \ell(C_d)\} =\{\ell(D_1),\ell(D_2),\ldots,\ell(D_d)\}$,
as multisets.
\item[{\em (b)}]
$\rank P=\max\{\ell(C_1),\ldots,\ell(C_d)\}$.
\item[{\em (c)}]
$\height (x)+\depth (x)=\rank \hat{P}$ for all $x\in C_i$ with $\ell(C_i)=\rank P$.
\end{enumerate}
\end{Corollary}

\begin{proof}
Let $\max(C_i)=\{x_i\}$ and $\max(D_i)=\{y_i\}$ for $i=1,\ldots,t$. We have alreday seen that
\[
\{x_1,x_2,\ldots,x_d\}=\{y_1,y_2,\ldots,y_d\},
\]
Therefore, the sets $$\{\height_P(x_1),\height_P(x_2),\ldots,\height_P(x_d)\}$$ and $$\{\height_P(y_1),\height_P(y_2),\ldots,\height_P(y_d)\}$$ are equal as multi-sets. By
Lemma~\ref{regular}, $\height_{P}(x_i)=\ell(C_i)$ and  $\height_{P}(y_i)=\ell(D_i)$. On the other hand, $\rank P=\max\{\height_P(x_1),\height_P(x_2),\ldots,\height_P(x_d)\}$. Then we have  proved (a) and (b).

In order to prove (c), we observe that
\begin{eqnarray*}
\rank \hat{P}&=&\ell(\hat{C_i})=\height_{\hat{C_i}}(x)+\depth_{\hat{C_i}}(x)\\
&\leq & \height (x)+\depth (x)\leq \rank \hat{P}.
\end{eqnarray*}
\end{proof}

In the next theorem we present the  characterization of the regular hyper-planar lattices which are  pseudo-Gorenstein.

\begin{Theorem}\cite[Theorem 3.3]{EHHSara}
\label{equallength}
Let $L$ be a regular hyper-planar lattice and $C_1\cup \ldots \cup C_d$  a canonical chain decomposition of $P$. Then $L$ is pseudo-Gorenstein if and only if all $C_i$ have the same length.
\end{Theorem}

\begin{proof}
Suppose all the chains $C_i$ have the same length. Then Corollary~\ref{ell} implies that $\ell(C_i)=\rank \hat{P}$ for all $i$. Let $x\in P$. Then $x\in C_i$ for some $i$, and hence $\height (x)+\depth (x)=\rank \hat{P}$, by Corollary~\ref{ell}. Therefore, by Theorem~\ref{classification}, $L$ is pseudo-Gorenstein.

Conversely, suppose that not all $C_i$ have the same length. Then Corollary~\ref{ell} implies that there exists  one $C_i$  with $\ell(C_i)<\rank P$. As in the proof of Theorem~\ref{classification} we consider the strictly order reversing function $v(x)=\depth (x)$ and $v'(x)=\rank \hat{P}-\height (x)$. Let $x=\max(C_i)$. Then $v(x)=1$ and, since $L$ is regular, $v'(x)=\rank \hat{P}-(\ell(C_i)+1)>\rank \hat{P}-\rank P-1=1$. This shows that $L$ is not pseudo-Gorenstein.
\end{proof}

\begin{Examples}{\em 
1. For the poset $P$ from Figure~\ref{expseudo}, the latice $L=\MI(P)$ is pseudo-Gorenstein since $P$ satisfies the condition of Theorem~\ref{equallength} and it is not Gorenstein since $P$ is not pure. 
\begin{figure}[hbt]
\begin{center}
\psset{unit=0.6cm}
\begin{pspicture}(-2,0)(4,4)
\rput(0,0){$\bullet$}
\rput(0,2){$\bullet$}
\rput(0,4){$\bullet$}
\rput(2,0){$\bullet$}
\rput(2,2){$\bullet$}
\rput(2,4){$\bullet$}
\psline(0,0)(0,4)
\psline(2,0)(2,4)
\psline(0,0)(2,4)
\end{pspicture}
\end{center}
\caption{}\label{expseudo}
\end{figure}

2. The lattice $L=\MI(P)$ where $P$ is the regular planar poset displayed in Figure~\ref{NotMy} is not pseudo-Gorenstein. 
}
\end{Examples}

The next theorem shows that, for regular planar lattices,  the necessary condition given in Theorem~\ref{alsoviviana} is also sufficient for the levelness of the Hibi ring. Before stating this theorem we need a preparatory result.

\begin{Lemma}
\label{biggerone}
Let  $L$ be  a regular planar lattice. Let $C_1\union C_2$ be a canonical chain decomposition of $P$, and  assume  that $\ell(C_1)=\rank P$ (cf.\ Corollary~\ref{ell}). Suppose that $P$ satisfies inequality (\ref{inequality}) given in Theorem~\ref{alsoviviana}. Then, for every minimal element $v\in \MT(P),$ we have $v(\max (C_1))=1$.
\end{Lemma}

\begin{proof} Let $v\in \MT(P)$ be a minimal element and
assume that $v(\max(C_1))>1$. Then $v(z)\geq \depth(z)+1$ for all $z\in C_1$.

Let
\[
v'(x)= \left\{ \begin{array}{ll}
       v(x)-1, & \;\textnormal{if $v(x)\geq \depth(x)+1$\;  (I),} \\ v(x), & \;\text{if $v(x)=\depth(x)$\hspace{0.8cm} (II),}
        \end{array} \right.
\]
for all $x\in \hat{P}$.

We show that $v'\in \MT(P)$ and  $v-v'\in \MS(P)$. Since $v'\neq v$, this will then show that $v$ is not minimal, a contradiction. Indeed, to see that  $v'\in \MT({P})$ we have to show that $v'(x)<v'(y)$ for all $x\gtrdot y$. If both $x$ and $y$ satisfy (I) or (II), then the assertion is trivial. If $x$ satisfies (I) and $y$ satisfies (II), then $v'(x)=v(x)-1<v(y)=v'(y)$, and if  $x$ satisfies (II) and $y$ satisfies (I), then $v(x)=\depth(x)\leq \depth(y)-1\leq v(y)-2$. Hence $v(x)<v(y)-1$, and this implies that $v'(x)<v'(y)$.

It remains to be shown that $v-v'\in \MS(P)$ which amounts to prove that $v(x)-v'(x)\leq v(y)-v'(y)$ for all $x\gtrdot y$. For this we only need to show that we cannot have $v'(x)=v(x)-1$ and $v(y)=v'(y)$, or, equivalently, that  $v(x)\geq \depth(x)+1$ and $v(y)=\depth(y)$ is impossible.

Assume to the contrary that there exist $x\gtrdot y$ with $v(x)\geq \depth(x)+1$ and $v(y)=\depth(y)$ . Then $y\not\in C_1$ since $v(z)\geq \depth(z)+1$ for all $z\in C_1$. Thus, we may either have $x\in C_1$ and $y\in C_2$, or $x,y\in C_2$.

In the first case, since  $\height(x)+\depth(y)\leq \rank \hat{P}+1$ by assumption, and since $\rank \hat{P}=\height(x)+\depth(x)$ due to the regularity of $L$  (see Corollary~\ref{ell}), we get $\depth(y)\leq \depth(x)+1\leq v(x)<v(y),$ a contradiction.

Finally, let $x,y\in C_2$. Since $v(x)<v(y)$, it follows that  $\depth(y)>\depth(x)+1$.  Therefore, the longest chain from $y$ to $\infty$ cannot pass through $x$. This implies that there exists $z\in C_1$ with $z\gtrdot y$. As in the first case, we then deduce that $v(y)>\depth(y)$. So we get again a contradiction.
\end{proof}

\begin{Theorem}
\label{converse}
Let  $L$ be  a regular planar  lattice.  Then the following conditions are equivalent:
\begin{enumerate}
\item[{\em (a)}] $L$ is level;
\item[{\em (b)}] $\height(x)+\depth(y)  \leq \rank \hat{P}+1$ for all $x,y\in P$ with $x\gtrdot y$;
\item[{\em (c)}]  for all $x,y\in P$ with $x\gtrdot y$, either  $\depth(y) = \depth(x)+1$  or $\height(x) = \height(y)+1$.
\end{enumerate}
\end{Theorem}

\begin{proof}
(a) \implies (b) follows from Theorem~\ref{alsoviviana}.

(b) \implies (c): Let $C_1\cup C_2$ be a canonical chain decomposition of $P$ with $|C_1|\geq |C_2|$.

If $x,y \in C_1$ or $x,y \in C_2$, then, by Lemma~\ref{regular}, it follows that $\height(x) = \height(y)+1$.

Next suppose that $x\in C_1$. Since $L$ is regular, we may apply Corollary~\ref{ell} and conclude that $\height(x)+\depth(x)=\rank \hat{P}$. Thus, by (b), we get $\depth (y)\leq \depth (x)+1$. On the other hand, it is clear that $\depth (y)\geq \depth (x)+1$. So that $\depth (y)=\depth (x)+1$.
Finally, if $y\in C_1$, then, by Corollary~\ref{ell}, we have $\height(y)+\depth(y)=\rank \hat{P}$. As in the previous case, we conclude that $\height(x) = \height(y)+1$.

(c)\implies (b): If $\depth(y)=\depth(x)+1$, then $\height(x)+\depth(y)=\height(x)+\depth(x)+1\leq \rank \hat{P}+1$, and if $\height (x)=\height(y)+1$, then $\height(x)+\depth(y)=\height (y)+\depth(y)+1\leq \rank \hat{P}+1$.

(b) \implies (a): As in Lemma~\ref{biggerone}  we let  $C_1\union C_2$ be a canonical chain decomposition of $P$, and  may assume  that $\ell(C_1)=\rank P\geq \ell(C_2)$. Let $v$ be minimal in $\MT(P)$. We will  show that there exists $v'\in \MT(P)$ with $v'(-\infty)=\rank \hat{P}$ and such that $v-v'\in \MS(P)$. Since $v$ is a minimal generator it follows that $v=v'$, thus $v(-\infty)=\rank\hat{P}.$ Consequently, it follows that all the minimal generators of $\omega_L$ have the same degree.

In order to construct $v'$ we consider the subposet $Q$ of $P$ which is obtained from $P$ by removing the maximal elements $\max(C_1)$ and $\max(C_2)$. We define on $\hat{Q}$ the strictly order reversing function $u$  by $u(\infty)=0$,  and $u(z)=v(z)-1$ for all other $z\in \hat{Q}$. We notice that the ideal lattice of $Q$ is again a regular planar  lattice satisfying (b). Indeed, assume that there exist $x\gtrdot y$ with $x,y\in Q$ such that $\height_{\hat{Q}}(x)+
\depth_{\hat{Q}}(y)>\rank \hat{Q}+1=\rank \hat{P}$. Since $\height_{\hat{Q}}(x)=\height(x)$ and $\depth(y)=\depth_{\hat{Q}}(y)+1$, it follows that
\[
\height(x)+\depth(y)=\height_{\hat{Q}}(x)+\depth_{\hat{Q}}(y)+1>\rank \hat{P}+1,
\]
a contradiction.

Therefore, by induction on the rank we may assume that the ideal lattice of $Q$ is level. Hence, there exists $w\in \MT(Q)$ with $w(-\infty)=\rank \hat{Q}=\rank \hat{P}-1$ and such that $u-w\in \MS(Q)$. Set $v'(z) =1+w(z)$ for all $z\in A=Q\union\{-\infty\}$. Then $v'$ is a strictly order reversing function on $A$ with $v'(-\infty)=\rank \hat{P}$ and such that $v-v'$ is order reversing on $A$. It remains to define $v'(C_i)$ for $i=1,2$ in a way such  that $v'\in \MT(P)$ and $v-v'\in \MS(P)$. We have to set $v'(\max(C_1))=1$ since $v(\max(C_1))=1$, and of course $v'(\infty)=0$. Let $x=\max(C_2)$ and let $z\in C_2$ be the unique element with $x\gtrdot z$. We set $v'(x)= v(x)-u(z)+w(z)= v(x)-v(z)+1+w(z)$, and claim that this $v'$ has the desired properties. Indeed, $v'(x)=v(x)-(v(z)-1-w(z))\leq v(x)$ and $v'(x)<1+w(z)=v'(z)$, since $v(x)<v(z)$. If $z$ is the only element covered by $x$, we are done. Otherwise, there exists $y\in C_1$ with $x\gtrdot y$ and it remains to be shown that $v'(y)>v'(x)=v(x)-v(z)+1+w(z)$. Suppose we know that $\depth_{\hat{Q}}(y)\geq w(z)$,  then
\[
v'(y)=w(y)+1\geq \depth_{\hat{Q}}(y)+1>w(z)\geq v'(x),
\]
as desired, since $v(x)-v(z)+1\leq 0$.  Thus, in order to complete the proof, we have to show that $\depth_{\hat{Q}}(y)\geq w(z)$. Since the ideal lattice of $Q$ is regular, this is equivalent to showing that
\begin{eqnarray}
\label{last}
w(z)\leq \rank \hat{Q}-\height_{\hat{Q}}(y).
\end{eqnarray}
The assumption  (b) and  Corollary~\ref{ell}(c)  imply that
\[
\height(x)+\depth(y)\leq \rank \hat{P}+1=\height(y)+\depth(y)+1,
\]
so that $\height(x)\leq \height(y)+1$. This yields
\begin{eqnarray}
\label{good}
\height(x)= \height(y)+1
\end{eqnarray}
since  $\height(x)\geq \height(y)+1$  always holds.

On the other hand, since $L$ is regular, Lemma~\ref{regular} implies that  $\height_P(x)= \height_{C_2}(x) =\height_{C_2}(z)+1=\height_{P}(z)+1$. This implies that $\height (x)=\height (z)+1$. So together with (\ref{good})  we then conclude that $\height(y)=\height(z)$. Since $\height_{\hat{Q}}(y)= \height(y)$ and $\height(z)= \height_{\hat{Q}}(z)$, inequality (\ref{last}) becomes $w(z)\leq \rank \hat{Q}-\height_{\hat{Q}}(z)$, and since $w(-\infty) =\rank \hat{Q}$,  this inequality indeed holds. This completes the proof of the theorem.
\end{proof}

\begin{Remark}{\em 
We do not know any example showing that condition (b) of the above theorem is not sufficient for the levelness of the Hibi ring. We conjecture that condition (b) in Theorem~\ref{converse} is also sufficient for any distributive lattice. 

\medskip

At the end of this subsection we go back to Hibi's examples of Figure~\ref{Hibiexample}. They correspond to the two posets displayed in Figure~\ref{Hibiposets}. 
It is easily seen that the left side poset which corresponds to the level lattice in Figure~\ref{Hibiexample} is not regular and of course
satisfies condition (b) in the above theorem. The right side poset corresponds to the non-level lattice in Figure~\ref{Hibiexample} and 
it does not satisfy condition (b) in Theorem~\ref{converse}.

\begin{figure}[hbt]
\begin{center}
\psset{unit=0.6cm}
\begin{pspicture}(-2,-2)(-4,4)
\rput(-8,0){
\rput(0,0){$\bullet$}
\rput(0,2){$\bullet$}
\rput(0,4){$\bullet$}
\rput(2,-2){$\bullet$}
\rput(2,0){$\bullet$}
\rput(2,2){$\bullet$}
\rput(2,4){$\bullet$}
\psline(0,0)(0,4)
\psline(2,-2)(2,4)
\psline(0,2)(2,0)
\psline(0,4)(2,2)
}
\rput(-1,0){
\rput(0,0){$\bullet$}
\rput(0,2){$\bullet$}
\rput(0,4){$\bullet$}
\rput(2,-2){$\bullet$}
\rput(2,0){$\bullet$}
\rput(2,2){$\bullet$}
\rput(2,4){$\bullet$}
\psline(0,0)(0,4)
\psline(2,-2)(2,4)
\psline(0,2)(2,0)
\psline(0,4)(2,2)
\psline(0,0)(2,4)
}
\end{pspicture}
\end{center}
\caption{}\label{Hibiposets}
\end{figure}

}
\end{Remark}

\subsection{Level and pseudo-Gorenstein  generalized Hibi rings}
\label{levelgeneralized}

In Subsection~\ref{generalized} we have presented the construction of the  generalized Hibi ring $R_r(P).$ Here $r\geq 2$ is an integer and 
$P$ is a finite poset. We have seen in Theorem~\ref{thm43} that $R_r(P)$ is the classical Hibi ring of the lattice $L_r=\MI_r(P)$ whose poset of join-irreducible elements is $P_r=P\times Q_{r-1}$. This identification allowed us to prove that the generalized Hibi ring is Gorenstein if and only if $P$ is pure. 

In the next theorem, following \cite[Section 5]{EHHSara}, we investigate some other properties of $R_r(P).$

\begin{Theorem}\cite[Theorem 5.1]{EHHSara}
Let $P$ be a finite poset and $r\geq 2$ an integer. Let $L=\MI(P)$ and $L_r=\MI(P_r)$. Then
\begin{itemize}
	\item [(a)] $\type R[L]\leq \type  R[L_r];$
	\item [(b)] The ring $R[L]$ is pseudo-Gorenstein if and only if $R[L_r]$ is pseudo-Gorenstein.
	\item [(c)] If $R[L_r]$ is level, then $R[L]$ is level.
\end{itemize}
\end{Theorem}

\begin{proof} (a) We know that, for a distributive lattice $L,$ $\type R[L]$ is equal to the number of the minimal generators of $\omega_
L,$ thus, $\type R[L]=|\min \MT(P)|$ where $\min \MT(P)$ denotes the set of minimal elements in $\MT(P).$ Therefore, in order to prove (a), 
it suffices to find an injective map $\min \MT(P)\to \min \MT(P_r).$ We define $\varepsilon: \min \MT(P)\to \min \MT(P_r)$ as follows. If 
$v\in\min \MT(P),$ then $\varepsilon(v)(x,i)=v(x)+(r-1-i)$ and $\varepsilon(v)(\infty)=0, \varepsilon(v)(-\infty)=v(-\infty)+(r-2).$
One easily checks that $v^\prime=\varepsilon(v)\in \MT(P_r).$ In order to show that $\varepsilon(v)\in \min\MT(P_r)$, we prove that if 
$u\in \MT(P_r)$ and $v^\prime -u \in \MS(P_r),$ then $v^\prime=u.$ 

For any $w\in \MT(P_r)$ and for $i\in [r-1]$ we define the function  $w_i$ on $\hat{P}$ as follows:
\[
w_i(x)=w(x,i)-(r-1-i) \quad \text{for all} \quad x\in P,
\]
$w_i(\infty)=0$ and $w_i(-\infty)=\max\{w_i(x)\: x\in P\}+1$.  Then $w_i\in \MT(P)$.

Since $v'-u\in  \MS(P_r)$ it follows that  $v-u_i=v'_i-u_i\in \MS(P)$. Since $v\in \min\MT(P)$ we get $v=u_i$ for all $i$. This shows that $v'=u$. 

It remain to show that $\varepsilon$ is injective. Let $v,w\in \min \MT(P)$ with $\varepsilon(v)=\varepsilon(w).$ By the definition of 
$\varepsilon$ we get immediately $v=w.$
 
(b) By Theorem~\ref{classification}, $R[L_r]$ is pseudo-Gorenstein if and only if, for all $x\in P,$ 
\[
\height_{\hat{P}_r} x+ \depth_{\hat{P}_r}x=\rank \hat{P}_r. 
\]

We will show that, for all $x\in P$ and $1\leq i\leq r-1,$ 
\begin{equation}\label{ht}
\height_{\hat{P}_r} x =\height_{\hat{P}} x+i-1 \text{ and } \depth_{\hat{P}_r} x =\depth_{\hat{P}} x+(r-i-1).
\end{equation}
In particular, we get $\rank \hat{P}_r=\rank \hat{P}+(r-2).$ These equalities will then imply that $R[L_r]$ is pseudo-Gorenstein if and 
only if $\height_{\hat{P}} x+\depth_{\hat{P}_r} x=\rank \hat{P} $, that is, if and only if $R[L]$ is pseudo-Gorenstein.

We will prove only the first equality in (\ref{ht}). The other one may be proved in a similar way. If $\height_{\hat{P}} x=1$, then 
we have nothing to prove since $x$ is a minimal element in $P$ and $i=1.$ Let $\height_{\hat{P}} x>1$ and 
$x=x_0>x_1>\cdots >x_h>-\infty$ be a maximal chain in $\hat{P}$ of length equal to  $\height_{\hat{P}} x.$ Then
\[
(x,i)=(x_0,i)>(x_1,i)>\cdots > (x_h,i)>(x_h,i-1)>\cdots > (x_h, 1)>-\infty
\] is a chain of length $\height_{\hat{P}} x+i-1$ in $\hat{P}_r.$ Therefore, we have $\height_{\hat{P}_r} x \geq\height_{\hat{P}} x+i-1.$ 
For the other inequality, we proceed by induction on $\height_{\hat{P}} x.$ Let $(x,i)=z_0>z_1>\cdots > z_t>-\infty$ be a chain of length 
$\height_{\hat{P}} x$ in $\hat{P}_r.$ Then either $z_1=(y,i)$ where $x\gtrdot y$ in $P$ or $z_1=(z,i-1)$. By the inductive hypothesis, in 
the first case $\height_{\hat{P}_r}(z_1)\leq \height_{\hat{P}} y+(i-1)\leq \height_{\hat{P}} x+i-2,$ and, in the second case,
$\height_{\hat{P}_r}(z_1)\leq \height_{\hat{P}} x+(i-2).$ In both cases it follows that $\height_{\hat{P}_r} x \leq\height_{\hat{P}} x+i-1. $

(c) Suppose  that $L$ is not level. Then there exists $v\in \min \MT(P)$ with $v(-\infty)>\rank \hat{P}$. Then $\varepsilon(v)$, as defined in the proof of part (a), belongs to  $\min \MT(P_r)$ and
\[
\varepsilon(v)(-\infty)=v(-\infty)+(r-2)>\rank \hat{P}+(r-2)=\rank \hat{P}_r.
\]
This shows that $L_r$ is not level.
\end{proof}

\section{The regularity of Hibi rings}
\label{regsection}

Let $L$ be a distributive lattice and $P$ its subset of join-irreducible elements. We assume that $|P|=n.$ Hence, $\rank L=n.$ 
The ring $R[L]$ is a standard graded algebra with the presentation $R[L]=S/I_L$ where $S=K[\{x_\alpha: \alpha\in L\}]$ and 
\[I_L=(x_\alpha x_\beta-x_{\alpha\cap\beta}x_{\alpha\cup\beta}: \alpha\in L, \alpha,\beta \text{ incomparable}).\]

Not so much is known about the $S$-resolution of the Hibi ring $R[L].$

One may easily compute the projective dimension of $R[L].$ Since $R[L]$ is Cohen-Macaulay, $\projdim R[L]=|L|-\dim R[L].$ Since 
$R[L]$ and $S/\ini_<(I_L)$ have the same Hilbert series, it follows that $\dim R[L]=\dim (S/\ini_<(I_L)).$ As we have already seen in Subsection~\ref{GBsubsect}, $S/\ini_<(I_L)$ is the Stanley-Resiner ring of the order complex of $L.$ Since the facets of this complex have the cardinality equal to $|P|+1,$ we get $\dim (S/\ini_<(I_L))=|P|+1.$ Therefore,
\[
\projdim R[L]=|L|-|P|-1.
\]

Another important homological invariant of $R[L]$ is the regularity. In this section we present the formula for $\reg R[L]$ following \cite{EHSara}. This can be given in terms of the poset $P.$ In the second part of this section, we study Hibi rings with linear syszygies and with pure resolution  for planar distributive lattices.

\subsection{The regularity of Hibi rings}
\label{regsubsection}

Before stating the formula for the regularity of $R[L],$ we explain how one may compute the regularity of a Cohen-Macaulay standard graded $K$--algebra. Let $R$ be a Cohen-Macaulay standard graded $K$--algebra, say $R=T/I$ where $T=K[x_1,\ldots,x_n]$ and 
$I\subset T$ a graded ideal. The Hilbert series of $R$ has the form $H_R(t)=Q(t)/(1-t)^{\dim R}$ where 
$Q(t)=\sum_{i\geq 0}h_it^i\in \ZZ[t]$ with $Q(1)\neq 0.$ The vector of the coefficients of $Q,$ $h=(h_0,h_1,\ldots)$, is called the 
{\em $h$--vector} of $R.$ As $R$ is Cohen-Macaulay, one may find an $R$--regular sequence $\theta=\theta_1,\ldots,\theta_{\dim R}$ of linear forms. The rings $R$ and $R/\theta R$ have the same $h$--vector and the same regularity \cite[Theorem 20.2]{P}. Since $\dim R/\theta R=0,$ we have $\reg R/\theta R=\deg h$ \cite[Exercise 20.18]{Eis2}. Consequently, 
\begin{equation}\label{regdeg}
\reg R=\deg h.
\end{equation}

The {\em $a$--invariant} $a(R)$ of $R$ is defined as the degree of the Hilbert series of $R;$ see \cite[Definition 4.4.4]{BHbook}. Hence, we have $a(R)=\deg h-\dim R.$ On the other hand, $a(R)=-\min\{i: (\omega_R)_i\neq 0\}$ where $\omega_R$ is the canonical module of $R$
 \cite[Chapter 4]{BHbook}. Therefore,
\begin{equation}\label{eqreg}
\reg R=\dim R-\min\{i: (\omega_R)_i\neq 0\}.
\end{equation}

\begin{Theorem}\cite[Theorem 1.1]{EHSara}\label{regularity}
Let $L=\MI(P)$ be a distributive lattice and $R[L]$ its Hibi ring. Then $\reg R[L]=|P|-\rank P-1.$
\end{Theorem}

\begin{proof}
We know that $\dim R[L]=|P|+1.$ By equality (\ref{eqreg}), we need to compute the initial degree of $\omega_L.$ In other words, we have 
to compute $\min\{v(-\infty): v\in \MT(P)\}.$ We have seen in Subsection~\ref{canonicalHibi} that $v(-\infty)\geq \rank\hat{P}=\rank P+2.$ On the other hand, $\depth\in \MT(P)$ and $\depth(-\infty)=\rank\hat{P}.$
 Therefore, 
$\min\{v(-\infty): v\in \MT(P)\}=\rank P+2.$ This implies that $\reg R[L]=|P|-\rank P-1.$
\end{proof}

A combinatorial proof of the above theorem can be found in \cite{EHSara}.

As a direct consequence of Theorem~\ref{regularity}, we may characterize the lattices $L$ for which $R[L]$ has a linear resolution. 
This characterization was first obtained in \cite{ERQ}. We can restrict to simple lattices. 
 Recall that $L=\MI(P)$ is called {\em simple} if there is no  $p \in P$ with the property that any element of $P$ is comparable to $p.$
In  lattice, this means that there are no elements $\alpha<\beta $ in $L$ such that any element $\gamma\in L$ satisfies either 
$\gamma\geq \beta$ or $\gamma\leq \alpha.$
 In what follows, we may assume without any restrictions that $L$ is simple. Indeed, if $L$ is not simple, we let  $P'$ to be the subposet  of $P$ which is obtained by removing a vertex $p\in P$ which is comparable to any other vertex of $P$ and set $L'=\MI(P')$.  Then $I_L$ and $I_{L'}$ have the same  regularity. Indeed, $|P'|=|P|-1$, and since any maximal chain of $P$ passes through $p$, it also follows that $\rank P'=\rank P-1$. Thus the assertion follows from Theorem~\ref{regularity}.

\begin{Corollary}
\label{linear resolution}
Let $L=\MI(P)$ be a finite simple distributive lattice. Then $R[L]$ has a linear resolution if and
only if $P$ is the direct sum  of a chain and an isolated element.
\end{Corollary}

\begin{proof}
The Hibi ring $R[L]$ has a linear resolution if and only if $\reg R[L]=1.$ By Theorem~\ref{regularity}, this is equivalent to 
$|P|-\rank P=2$. Hence, apart of a chain, $P$ contains just one element. This element cannot be comparable to any element of the chain since the lattice is simple. 
\end{proof}

Theorem~\ref{regularity} allows the characterization of several other Hibi rings. 

Extremal Cohen-Macaulay and Gorenstein algebras appeared in \cite{Sa} and \cite{Sch}. In \cite{KSK}, nearly extremal Cohen-Macaulay and Gorenstein algebras were defined. Let $R=T/I$ be a standard graded algebra. Here $T$ is a polynomial ring over $K$ n finitely many variables and $I\subset T$ is a graded ideal. Let $h=(h_0,\ldots,h_s)$ be the $h$--vector of $R$ and assume that the initial degree 
of $I$ is equal to $p.$

(1). Suppose that $R$ is Cohen-Macaulay. Then $s\geq p-1.$ If $s=p-1$ ($s=p$), then $R$ is called ({\em nearly}) {\em extremal Cohen-Macaulay}.

(2). Suppose that $R$ is Gorenstein. Then $s\geq 2(p-1)$. If $s=2(p-1)$ (s=2p-1), then $R$ is called ({\em nearly}) {\em extremal  Gorenstein}.

Since $\reg R[L]=\deg h,$ we may use Theorem~\ref{regularity} to characterize the simple lattices $L$ (or, equivalently, the poset $P$) for 
which $R[L]$ is a (nearly) extremal Cohen-Macaulay or Gorenstein algebra. In our case, the initial degree of the presentation ideal 
of $R[L]$ is equal to $2.$ Therefore, we get:

(i). If $\reg R[L]=1$ ($\reg R[L]=2$), then $R[L]$ is (nearly) extremal Cohen-Macaulay. Thus, $R[L]$ is extremal Cohen-Macaulay if and only if $R[L]$ has a linear resolution. 

In order to characterize the lattices $L$ for which $\reg R[L]=2,$ we have to find all the posets $P$ with $|P|=\rank P +3.$ This characterization was done in \cite{EHSara}. Let $C$ be a maximal chain in $P$. Since $|P|=\rank P+3$, it follows that there exist precisely two elements $q,q^\prime\in P$ which do not belong to $C$. The only posets satisfying  $|P|=\rank P+3$ for which $L=\MI(P)$ is  simple are displayed in Figure~\ref{reg=3}.

\begin{figure}[hbt]
\begin{center}
\psset{unit=0.5cm}
\begin{pspicture}(-2,-1)(4,7)
\rput(-12,0){
\rput(0,0){$\bullet$}
\rput(1,0){$\bullet$}
\rput(2,0){$\bullet$}
\rput(2,2){$\bullet$}
\rput(2,4){$\bullet$}
\rput(2,6){$\bullet$}
\psline(2,0)(2,2)
\psline[linestyle=dotted](2,2)(2,4)
\psline(2,4)(2,6)
}
\rput(-7,0){
\rput(0,0){$\bullet$}
\rput(1,0){$\bullet$}
\rput(2,0){$\bullet$}
\rput(2,2){$\bullet$}
\rput(2,4){$\bullet$}
\rput(2,6){$\bullet$}
\psline(2,0)(2,2)
\psline[linestyle=dotted](2,2)(2,4)
\psline(2,4)(2,6)
\psline(1,0)(2,3)
}
\rput(-2,0){
\rput(0,2){$\bullet$}
\rput(0,0){$\bullet$}
\rput(2,0){$\bullet$}
\rput(2,2){$\bullet$}
\rput(2,4){$\bullet$}
\rput(2,6){$\bullet$}
\psline(2,0)(2,2)
\psline[linestyle=dotted](2,2)(2,4)
\psline(2,4)(2,6)
\psline(0,0)(0,2)
}
\rput(3,0){
\rput(0,4){$\bullet$}
\rput(0,0){$\bullet$}
\rput(2,0){$\bullet$}
\rput(2,2){$\bullet$}
\rput(2,4){$\bullet$}
\rput(2,6){$\bullet$}
\psline(2,0)(2,2)
\psline[linestyle=dotted](2,2)(2,4)
\psline(2,4)(2,6)
\psline(0,0)(2,3)
\psline(0,0)(0,4)
}
\rput(8,0){
\rput(0,4){$\bullet$}
\rput(0,0){$\bullet$}
\rput(2,0){$\bullet$}
\rput(2,2){$\bullet$}
\rput(2,4){$\bullet$}
\rput(2,6){$\bullet$}
\psline(2,0)(2,2)
\psline[linestyle=dotted](2,2)(2,4)
\psline(2,4)(2,6)
\psline(0,4)(2,3)
\psline(0,0)(0,4)
}
\rput(13,0){
\rput(0,4){$\bullet$}
\rput(0,0){$\bullet$}
\rput(2,0){$\bullet$}
\rput(2,2){$\bullet$}
\rput(2,4){$\bullet$}
\rput(2,6){$\bullet$}
\psline(2,0)(2,2)
\psline[linestyle=dotted](2,2)(2,4)
\psline(2,4)(2,6)
\psline(0,4)(2,2.5)
\psline(0,0)(2,3.5)
\psline(0,0)(0,4)
}
\end{pspicture}
\end{center}
\caption{}\label{reg=3}
\end{figure}

(ii). Let $R[L]$ be Gorenstein. By Theorem~\ref{GorensteinHibi}, $P$ is a pure poset. The ring $R[L]$ is (nearly) extremal Goresnstein 
if $\reg R[L]=2$ ($\reg R[L]=3$). In the first case we get easily the poset $P$ of one of the forms displayed in Figure~\ref{extremalGor};
see also \cite{EHSara}.
In the second case one obtains again a finite number of posets $P$ for which $R[L]$ is nearly extremal Gorenstein.

\begin{figure}[hbt]
\begin{center}
\psset{unit=0.5cm}
\begin{pspicture}(-4,-1)(4,3)
\rput(-12,0){
\rput(0,0){$\bullet$}
\rput(1.5,0){$\bullet$}
\rput(3,0){$\bullet$}
}
\rput(-5,0){
\rput(0,0){$\bullet$}
\rput(2,0){$\bullet$}
\rput(0,2){$\bullet$}
\rput(2,2){$\bullet$}
\psline(0,0)(0,2)
\psline(0,0)(2,2)
\psline(2,0)(0,2)
\psline(2,0)(2,2)
}
\rput(2,0){
\rput(0,0){$\bullet$}
\rput(2,0){$\bullet$}
\rput(0,2){$\bullet$}
\rput(2,2){$\bullet$}
\psline(0,0)(0,2)
\psline(2,0)(0,2)
\psline(2,0)(2,2)
}
\rput(9,0){
\rput(0,0){$\bullet$}
\rput(2,0){$\bullet$}
\rput(0,2){$\bullet$}
\rput(2,2){$\bullet$}
\psline(0,0)(0,2)
\psline(2,0)(2,2)
}
\end{pspicture}
\end{center}
\caption{Extremal Gorenstein}\label{extremalGor}
\end{figure}
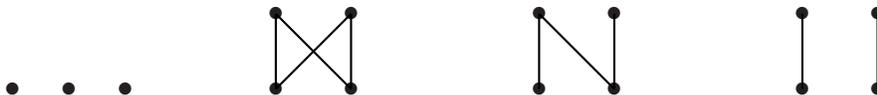

\medskip
We end this subsection by a few comments on the regularity of Hibi rings for planar distributive lattices. 
We consider the infinite distributive lattice  $\NN^2$ with the partial order  defined as 
$(i,j)\leq (k,\ell)$ if $i\leq k$ and $j\leq \ell.$  A {\em planar distributive lattice} is a finite sublattice $L$ of $\mathbb{N}^2$ with $(0,0)\in L$ which has the following property: for any $(i,j), (k,\ell)\in L$ there exists a chain $\cc$ in $L$ of the form $\cc: x_0<x_1<\cdots <x_t$ with $x_s=(i_s,j_s)$ for 
$0\leq s\leq t,$ $(i_0,j_0)=(i,j)$, and $(i_t,j_t)=(k,\ell)$, such that $i_{s+1}+j_{s+1}=i_s+j_s+1$ for all $s.$ Planar distributive lattice are also called {\em two-sided ladders}.

In the planar case, we may compute the regularity of $R(L)$ in terms of the cyclic sublattices of $L$. A sublattice of $L$ is called {\em cyclic} if it looks like in  Figure~\ref{cyclic} with some possible cut edges in between the squares. By a {\em square} in $L$ we mean a sublattice with elements $a,b,c,d$ such that 
$d\gtrdot b\gtrdot a$, $d\gtrdot c\gtrdot a$, and $b,c$ are incomparable. A {\em cut edge} of the lattice $L$ is an edge $\beta\gtrdot \alpha$ in its Hasse diagram with the property that, for every $\gamma\in L,$ we have either $\gamma\geq \beta$ or $\gamma\leq \alpha.$

\begin{figure}[hbt]
\begin{center}
\psset{unit=0.6cm}
\begin{pspicture}(1,-2)(5,5)
\rput(0,-1){
\rput(0,1){\pspolygon(2,2)(3,3)(4,2)(3,1)
\rput(2,2){$\bullet$}
\rput(3,3){$\bullet$}
\rput(4,2){$\bullet$}
\rput(3,1){$\bullet$}
}
\rput(0,3){\pspolygon(2,2)(3,3)(4,2)(3,1)
\rput(2,2){$\bullet$}
\rput(3,3){$\bullet$}
\rput(4,2){$\bullet$}
\rput(3,1){$\bullet$}
}
\psline(3,2)(3,1)
\rput(0,0){
\pspolygon(3,1)(2,0)(3,-1)(4,0)
\rput(3,1){$\bullet$}
\rput(2,0){$\bullet$}
\rput(3,-1){$\bullet$}
\rput(4,0){$\bullet$}
}
}
\end{pspicture}
\end{center}
\caption{Cyclic sublattice}
\label{cyclic}
\end{figure}
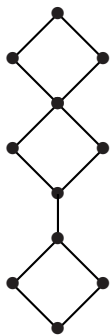

It is easily seen that, for a cyclic lattice $C$ with $r$ squares, we have $\reg R[C]=r.$
Of course, this may be derived with the formula of Theorem~\ref{regularity}, but we may give also a simpler argument as in  \cite{ERQ}. 
The ideal $I_C$ is generated by a regular sequence of length $r$ since $\ini_{<}(I_C)$ is generated by a regular sequence of monomials. Therefore, the Koszul complex of the generators of $I_C$ is the minimal free resolution of $R[C]$ and, hence, $\reg R[C]=r.$ 

\begin{Theorem}\cite{ERQ}\label{planar}
Let $L$ be a planar distributive lattice. Then $\reg R[L]$ equals the maximal number of squares in  a cyclic sublattice of $L.$
\end{Theorem}

The interested reader may find the complete proof in \cite{ERQ}. Here we only mention that the proof uses combinatorial interpretations of the components of the $h$-vector of $R[L]$ given in \cite[Section 2]{BGS}. It turns out that $\deg h$ is equal to the maximal numbers of squares in a cyclic sublattice of $L$ which explains the statement of the theorem.

The above theorem allows us, in relatively small examples,  to {\em read} the regularity of $R[L]$ by looking at the Hasse diagram of $L$ as
 in Figure~\ref{exampleplanar}.

\begin{figure}[hbt]
\begin{center}
\psset{unit=0.6cm}
\begin{pspicture}(3,-2)(5,5)
\rput(-4,-1){
\rput(0,0){$\bullet$}
\rput(2,0){$\bullet$}
\rput(4,0){$\bullet$}
\rput(6,0){$\bullet$}
\rput(0,2){$\bullet$}
\rput(2,2){$\bullet$}
\rput(4,2){$\bullet$}
\rput(6,2){$\bullet$}
\rput(4,4){$\bullet$}
\rput(6,4){$\bullet$}
\rput(4,6){$\bullet$}
\rput(6,6){$\bullet$}
\pspolygon(0,0)(6,0)(6,2)(0,2)
\pspolygon[fillstyle=crosshatch*](2,0)(4,0)(4,2)(2,2)
\pspolygon[fillstyle=crosshatch*](4,4)(6,4)(6,6)(4,6)
\pspolygon(4,0)(6,0)(6,6)(4,6)
\pspolygon[style=fyp,fillcolor=light](0,0)(2,0)(2,2)(0,2)
\pspolygon[style=fyp,fillcolor=light](4,2)(6,2)(6,4)(4,4)
\rput(3,-1){$\reg R[L]=2$}
}
\rput(6,-1){
\rput(0,0){$\bullet$}
\rput(2,0){$\bullet$}
\rput(4,0){$\bullet$}
\rput(6,0){$\bullet$}
\rput(0,2){$\bullet$}
\rput(2,2){$\bullet$}
\rput(4,2){$\bullet$}
\rput(6,2){$\bullet$}
\rput(4,4){$\bullet$}
\rput(6,4){$\bullet$}
\rput(4,6){$\bullet$}
\rput(6,6){$\bullet$}
\rput(2,4){$\bullet$}
\pspolygon(0,0)(6,0)(6,2)(0,2)
\pspolygon(4,0)(6,0)(6,6)(4,6)
\pspolygon[style=fyp,fillcolor=light](0,0)(2,0)(2,2)(0,2)
\pspolygon[style=fyp,fillcolor=light](4,4)(6,4)(6,6)(4,6)
\pspolygon[style=fyp,fillcolor=light](2,2)(4,2)(4,4)(2,4)
\rput(3,-1){$\reg R[L]=3$}
}
\end{pspicture}
\end{center}
\caption{}
\label{exampleplanar}
\end{figure}

One could ask  whether we can {\em read} as well the pseudo-Gorenstein property of $R[L]$ from the drawing of $L.$ A rigorous 
answer to this question was given in \cite{EHHSara}. Here, we briefly explain the picture of the pseudo-Gorensteiness 
without giving a formal proof. As we have seen in Proposition~\ref{echivalent}, $R[L]$ is pseudo-Gorenstein if and only if the leading 
coefficient of the numerator of the Hilbert series of $R[L]$ is equal to $1.$ According to the proof of \cite[Theorem 4]{ERQ}, this coefficient is equal to the number of cyclic sublattices of $L$ with the largest number of squares. Hence, $L$ (or $R[L]$) is pseudo-Gorenstein if and only if it contains exactly one cyclic sublattice with maximum number of squares. For example, the lattice of the left side in Figure~\ref{exampleplanar} is not pseudo-Gorenstein since, as we may see in figure,  there are at least two cyclic sublattices with two squares, while the lattice displayed in the right side of the same figure is pseudo-Gorenstein.

\subsection{Hibi ideals with linear relations}
\label{subsectlinrel}

In the remaining part of this section we will restrict to planar distributive lattices. Even with this restriction, the calculation of all the graded Betti numbers of the Hibi ideals seems to be very difficult. In this subsection we aim at describing the shape of those planar distributive lattices $L$ with the property that $I_L$ has linear relations. We say that $I_L$ has {\em linear relations} or that it is {\em linearly related} if $\beta_{1j}(I_L)=0$ for all $j\geq 4.$

The following lemma offers a major reduction in our study; see also \cite[Corollary 1.4]{EHH}.
\begin{Lemma}
Let $I\subset T$ be a graded ideal in the polynomial ring $T$ over a field $K$ with  finitely many indeterminates. If $I$ has a quadratic Gr\"obner basis, then $\beta_{1j}(I)=0$ for $j>4.$
\end{Lemma}

\begin{proof}
By hypothesis, there exists a monomial order $<$ on $T$ such that $\ini_<(I)$ is generated in degree $2.$ Therefore, it follows  from \cite[Corollary 4]{HS} that
$\beta_{1j}(\ini_<(I))=0$ for $j>4$. Since $\beta_{1j}(I)\leq \beta_{1j}(\ini_<(I))$ (see, for example, \cite[Corollary 3.3.3]{HH10}), the desired conclusion follows.
\end{proof}

Almost all planar lattices may be viewed as convex polyominoes. For more information on this notion we refer the reader to \cite{EHH}. 
All convex polyominoes whose ideals have linear relations were characterized in \cite{EHH}. In this work, we follow the approach from 
\cite{EHH}, but we adapt some of the proofs to  Hibi ideals for planar lattices. The main tool in our study is the squarefree divisor complex
which allows  the calculation of the multi-graded Betti numbers of a toric ideal. 

We briefly recall the construction of the squarefree divisor complex which was introduced in \cite{BH}. Let $K$ be a field and $H\subset \NN^n$ an affine semigroup minimally generated by $h_1,\ldots,h_m$ where $h_i=(h_i(1),\ldots,h_i(n))\in \NN^n$. Let $K[H]\subset T= K[t_1,\ldots,t_n]$ be the semigroup ring assciated with $H.$ Then $K[H]=K[u_1,\ldots,u_m]$ where $u_i=\prod_{j=1}^n t_j^{h_i(j)}.$ 
Let $\varphi:R=K[x_1,\ldots, x_n]\to T$ be the $K$--algebra homomorphism induced by $x_i\mapsto u_i$ for $1\leq i
\leq n,$  and $I_H$ the kernel of $\varphi.$ The ideal $I_H$ is called the {\em toric ideal} of $K[H]$ or, simply, of $H.$ To each variable 
$x_i$ we assign the multi-degree $h_i.$ In this way, $K[H]$ and its toric ideal are $\ZZ^n$--graded $R$--modules. Thus $I_H$ and $K[H]$ 
 have $\ZZ^n$--graded minimal free resolutions. When all the monomials $u_i$ have the same degree, then $K[H]$ may be viewed as a standard 
 graded $K$--algebra by setting $\deg u_i=1$ for all $h.$ In this case, the degree of $t_1^{h(1)}\cdots t_n^{h(n)}\in K[H]$ will be denoted $|h|.$

Let $h\in H.$ The squarefree divisor complex $\Delta_h$ is defined as follows. Its facets are the sets $F=\{i_1,\ldots,i_k\}\subset [m]$
such that $u_{i_1}\cdots u_{i_k}| t_1^{h(1)}\cdots t_n^{h(n)}$ in $K[H].$ Let $\tilde{H}_i(\Gamma, K)$ be the i$^{th}$ reduced simplicial 
homology of a simplicial complex $\Gamma.$\footnote{For more information on the theory of simplicial complexes and their simplicial homology we refer the reader to \cite{Sta2} and \cite[Chapter 5]{BHbook}.}

\begin{Proposition}[\cite{BH}]
\label{bh}
With the notation and assumptions introduced one has  $$\Tor_i(K[H],K)_h\iso\tilde{H}_{i-1}(\Delta_h, K).$$ In particular,
$$\beta_{ih}(K[H])=\dim_K\tilde{H}_{i-1}(\Delta_h, K).$$
\end{Proposition}

Let $H'$ be a subsemigroup of $H$ generated by a subset of the set of generators  of $H$, and let $R'$ be the polynomial ring over $K$ in the variables $x_i$ with $h_i$ generator of $ H^\prime$. Furthermore, let $\FF'$ be the $\ZZ^{n}$-graded  free $R'$-resolution  of $K[H']$. Then, since $R$ is a flat $R'$-module,  $\FF'\tensor_{S'}S$ is a $\ZZ^n$-graded free $S$-resolution of $S/I_H'S$. The inclusion $K[H']\to K[H]$ induces a $\ZZ^n$-graded complex homomorphism $\FF'\tensor_{S'}S\to \FF$. Tensoring this complex homomorphism with $K=R/\mm$, where $\mm$ is the graded maximal ideal of $R$,  we obtain the following sequence of isomorphisms and  natural maps of $\ZZ^n$-graded $K$-modules
\[
\Tor_i^{R'}(K[H'],K)\iso H_i(\FF'\tensor_{R'}K)\iso H_i(\FF'\tensor_{R'}R)\tensor_RK)\to H_i(\FF\tensor_RK)\iso \Tor_i^R(K[H],K).
\]

With an additional assumption on $H^\prime$ we get even more.

\begin{Corollary}\cite[Corollary 2.3]{EHH}
\label{refinement}
With the notation and assumptions introduced, let  $H'$ be a subsemigroup of $H$ generated by a subset of the set of generators  of $H$, and let  $h$ be an
element of $H'$ with the property  that $h_i\in H'$ whenever  $h-h_i\in H$. Then  the natural $K$-vector space homomorphism $\Tor_i^{R'}(K[H'],K)_h\to
\Tor_i^R(K[H],K)_h$ is an isomorphism for all $i$.
\end{Corollary}

For the proof of this corollary  we refer to \cite{EHH}.

\begin{Definition}
Let $H\subset \NN^n$ be an affine semigroup generated by $h_1,\ldots, h_m$. An affine subsemigroup $H'\subset H$ generated by  a subset of $\{h_1,\ldots, h_m\}$ is called a {\em homological pure} subsemigroup of $H$ if for all $h\in H'$ and all $h_i$ with $h-h_i\in H$  it follows that $h_i\in H'$.
\end{Definition}

In other words, $H^\prime$ is a homological pure subsemigroup of $H$ if it satisfies the hypothesis of Corollary~\ref{refinement}.
We also say that $K[H^\prime]$ is a {\em homological pure subring} of $K[H].$
 Corollary~\ref{refinement} has the following consequence.

\begin{Corollary}\cite[Corollary 2.4]{EHH}
\label{homologicallypure}
Let $H'$ be a homologically pure subsemigroup  of $H$. Then $$\Tor_i^{R'}(K[H'],K)\to  \Tor_i^R(K[H],K)$$   is injective for all $i$. In other words, if $\FF'$ is the minimal $\ZZ^n$-graded free $R'$-resolution of $K[H']$ and $\FF$ is the minimal $\ZZ^n$-graded free $R$-resolution of $K[H]$, then the complex homomorphism
$\FF'\tensor R\to \FF$ induces an injective map $\FF'\tensor K\to \FF\tensor K$. In particular, any minimal set of generators of $\Syz_i(K[H'])$ is part of a minimal set of generators  of  $\Syz_i(K[H])$. Moreover, $\beta_{ij}(I_{H'})\leq \beta_{ij}(I_H)$ for all $i$ and $j$.
\end{Corollary}

For the proof, see \cite{EHH}.

Let $L$ be a planar distributive lattice. We may assume that $[(0,0),(m,n)]$ where $m,n$ are some positive integers, is the smallest interval of $\NN^2$ which contains $L.$ In particular, we may assume that $L$ contains the vertices of the squares $[(0,0),(1,1)]$ and $[(m-1,n-1),(m,n)].$ There is no loss of generality in this latter assumption since it 
 simply means that that the poset $P$ of the join-irreducible elements of $L$ has  two minimal and two maximal elements. 
If $P$ has a unique minimal element, say $p,$ then $R[\MI(P)]$ and $R[\MI(P\setminus\{p\})]$ have the same Betti numbers. The same happens when $P$ contains a unique maximal element.
We also may assume that $m,n\geq 2.$ If, for instance,  $n=1,$ then we know, by Theorem~\ref{linear resolution}, that $I_L$ has a linear resolution, thus, in particular, it has linear relations. 

Let $A=\{i_1,\ldots,i_s\}\subset [m]$ be a set of integers and let $L_A$ be the subset of $L$ obtained by removing all the elements $(i,j)$ of
$L$ with $i\in A.$ Then $L_A$ is a sublattice of $L.$ Indeed, if $(i,j),(k,\ell)\in L_A,$ then $i,k\notin A,$ thus $\min\{i,k\}$  and 
$\max\{i,k\}$ do not belong to $A$ as well. Thus $L_A$ is a sublattice of $L.$ Analogously, we may consider the same procedure by using a subset $B\subset [n]$ and get a sublattice $L_B$ of $L.$ We call a sublattice of $L$ obtained in one of the above ways an {\em induced sublattice} of $L.$ Moreover, one easily sees that if $L^\prime$ is an induced sublattice of $L,$ then $R[L^\prime]$ is a homological pure subring of $R[L].$ On the other hand, let us note that not any sublattice of $L$ is an induced one.

Corollary~\ref{homologicallypure} has the following consequence. 

\begin{Corollary}\label{embed}
Let $A\subset [m]$ and $B\subset [n].$ With the above notation, we have
\[
\beta_{ij}(I_{L_A})\leq \beta_{ij}(L) \text{ and } \beta_{ij}(I_{L_B})\leq \beta_{ij}(L)
 \] for all $i,j.$
Moreover, each minimal relation of $I_{L_A}$ or $I_{L_B}$ is a minimal relation of $I_L.$
\end{Corollary}

This corollary will be useful to isolate the Hibi ideals of planar lattices which have linear relations.

We begin with the following lemma which shows, in particular,  that in order to get linear relations for $I_L$ it is enough to consider 
$L$ a simple lattice.

\begin{Lemma}\label{tensor}
Let $L=\MI(P)\subset [(0,0),(m,n)]$ be a planar distributive lattice which contains the vertices of the squares $[(0,0),(1,1)]$ and $[(m-1,n-1),(m,n)].$ If $L$ is not simple, then $\beta_{14}(I_L)\neq 0.$
\end{Lemma}

\begin{proof}
The claim of the lemma is equivalent to $\beta_{24}(R[L])\neq 0.$ Since $L$ is not simple, there exists $p\in P$ such that any other element of $P$ is comparable to $p.$ Let $P_1=\{q\in P: q<p\}$ and $P_2=\{q\in P: q>p\}.$ Then $P$ is the ordinal sum $P=P_1\oplus \{p\}\oplus P_2$ and $R[\MI(P)]\cong R[\MI(P_1)]\otimes R[\MI(P_2)].$  Let $\FF_1\to R[\MI(P_1)]\to 0$ and $\FF_2\to R[\MI(P_2)]\to 0$ 
be the minimal free resolutions of $R[\MI(P_1)]$ and $R[\MI(P_2)].$ Then the total complex of $\FF_1\otimes \FF_2$ is the minimal 
$S$--free resolution of $R[L]$. This implies that $\beta_{24}(R[L])\neq 0$ since $\beta_{12}(R[\MI(P_1)])\neq 0$ and 
$\beta_{12}(R[\MI(P_2)])\neq 0.$
\end{proof}

The above lemma combined with Corollary~\ref{homologicallypure} lead to the following type of arguments. Assume that, given a simple planar distributive lattice $L,$ we may find a subset $A\subset [m]$  such that $L_A$ is not simple and contains the extremal corners $[(0,0),(1,1)]$ and $[(m-1,n-1),(m,n)]$. Then, it follows that $I_{L_A}$ is not linearly related. This will imply that $I_L$ is not linearly related, too. 

The following theorem characterizes the simple planar distributive lattices $L$ with linearly related Hibi ideals for $m,n\geq 3$. The case $m=2$ or $n=2$ is settled by the following lemma.

\begin{Lemma}\label{lema0}
Let $L$ be a simple planar distributive lattice $L\subset [(0,0),(m,n)]$ and assume that $m=2$ or $n=2.$ If $I_L$ is linearly related, then at most one of the vertices $(m,0)$ and $(0,n)$ do not belong to $L.$
\end{Lemma}

\begin{proof}
Let us take, fir example, $n=2.$ 
\begin{figure}[hbt]
\begin{center}
\psset{unit=0.5cm}
\begin{pspicture}(3,-1)(0,3)
\rput(-1,-1){
\rput(0,0){$\bullet$}
\rput(2,0){$\bullet$}
\rput(0,2){$\bullet$}
\rput(2,2){$\bullet$}
\rput(4,2){$\bullet$}
\rput(4,4){$\bullet$}
\rput(2,4){$\bullet$}
\pspolygon(0,0)(2,0)(2,2)(0,2)
\pspolygon(2,2)(4,2)(4,4)(2,4)
}

\end{pspicture}
\end{center}
\caption{}
\label{prooflemma0}
\end{figure}
If both vertices  $(m,0)$ and $(0,2)$ do not belong to $L,$ then we find an induced sublattice  of the form displayed in Figure~\ref{prooflemma0} which has the associated  ideal  not linearly related. 
\end{proof}

\begin{Theorem}\label{linrel}
Let $L$ be a simple planar distributive lattice, $L\subset [(0,0),(m,n)]$ with $m,n\geq 2$. The ideal $I_L$ is linearly related if and only if  the following conditions hold:
\begin{itemize}
	\item [(i)] At most one of the vertices $(m,0)$ and $(0,n)$ does not belong to $L.$
	\item [(ii)] The vertices $(1,n-1)$ and $(m-1,1)$ belong to $L.$
\end{itemize}
\end{Theorem}

The only if part of the proof of this theorem follows from the following lemmas.

\begin{Lemma}\label{lema1}
Let $L$ be a lattice as in the statement of Theorem~\ref{linrel} and assume that both vertices $(m,0)$ and $(0,n)$ do not belong to 
$L.$ Then $\beta_{14}(I_L)\neq 0.$
\end{Lemma}

\begin{proof} 
Let $A=\{2,\ldots,m-2\}$ and $L_A$ the corresponding induced lattice. The lattice $L_A$ may be now framed in the interval $[(0,0),(3,n)].$
We choose now the set $B=\{2,\ldots,n-2\}\subset [n]$ and consider the induced sublattice $L_{AB}$ of $L_A$. The lattice $L_{AB}$ is isomorphic to one of the form displayed in Figure~\ref{LAB}:
\begin{figure}[hbt]
\begin{center}
\psset{unit=0.5cm}
\begin{pspicture}(9,-2)(5,5)
\rput(-4,-1){
\rput(0,0){$\bullet$}
\rput(2,0){$\bullet$}
\rput(0,2){$\bullet$}
\rput(2,2){$\bullet$}
\rput(3,3){$\bullet$}
\rput(5,5){$\bullet$}
\rput(3,5){$\bullet$}
\rput(5,3){$\bullet$}
\pspolygon(0,0)(2,0)(2,2)(0,2)
\pspolygon(3,3)(5,3)(5,5)(3,5)
}
\rput(3,-1){
\rput(0,0){$\bullet$}
\rput(2,0){$\bullet$}
\rput(0,2){$\bullet$}
\rput(2,2){$\bullet$}
\rput(0,4){$\bullet$}
\rput(2,4){$\bullet$}
\rput(4,4){$\bullet$}
\rput(6,6){$\bullet$}
\rput(6,4){$\bullet$}
\rput(4,6){$\bullet$}
\pspolygon(0,0)(2,0)(2,4)(0,4)
\pspolygon(4,4)(6,4)(6,6)(4,6)
\psline(0,2)(2,2)
\psline(2,4)(4,4)
}
\rput(11,-1){
\rput(0,0){$\bullet$}
\rput(2,0){$\bullet$}
\rput(0,2){$\bullet$}
\rput(2,2){$\bullet$}
\rput(0,4){$\bullet$}
\rput(2,4){$\bullet$}
\rput(4,4){$\bullet$}
\rput(6,6){$\bullet$}
\rput(6,4){$\bullet$}
\rput(4,6){$\bullet$}
\rput(6,2){$\bullet$}
\rput(4,2){$\bullet$}
\pspolygon(0,0)(2,0)(2,4)(0,4)
\pspolygon(4,2)(6,2)(6,6)(4,6)
\pspolygon(2,2)(4,2)(4,4)(2,4)
\psline(0,2)(2,2)
\psline(4,4)(6,4)
}
\end{pspicture}
\end{center}
\caption{}
\label{LAB}
\end{figure}

In the first two cases, it is clear, by Lemma~\ref{tensor}, that $\beta_{14}(I_{L_{AB}})\neq 0.$ In the last case, one may easily see that $L_{AB}$ contains an induced cyclic sublattice with two squares, thus 
 $\beta_{14}(I_{L_{AB}})\neq 0.$ Hence, by applying Corollary~\ref{embed}, we get $\beta_{14}(I_L)\neq 0.$
\end{proof}

\begin{Lemma}\label{lema2}
Let $L$ be a lattice as in the statement of Theorem~\ref{linrel} and assume that $(m,0)\in L$ and $(0,n)\notin L$. If $I_L$ is linearly related, then $(1,n-1)\in L.$
\end{Lemma}

\begin{proof}
Assume that $(1,n-1)\notin L.$ We show that $\beta_{14}(I_L)\neq 0.$ Proceeding as in the proof of Lemma~\ref{lema1}, we get an induced sublattice $L_{AB}$ of $L$ which is displayed in Figure~\ref{latticelema2}. 

\begin{figure}[hbt]
\begin{center}
\psset{unit=0.5cm}
\begin{pspicture}(9,-2)(5,4)
\rput(4,-2){
\rput(0,0){$\bullet$}
\rput(2,0){$\bullet$}
\rput(0,2){$\bullet$}
\rput(2,2){$\bullet$}
\rput(4,0){$\bullet$}
\rput(4,2){$\bullet$}
\rput(4,4){$\bullet$}
\rput(6,0){$\bullet$}
\rput(6,4){$\bullet$}
\rput(4,6){$\bullet$}
\rput(6,2){$\bullet$}
\rput(6,6){$\bullet$}
\pspolygon(0,0)(6,0)(6,2)(0,2)
\pspolygon(4,0)(6,0)(6,6)(4,6)
\psline(2,0)(2,2)
\psline(4,4)(6,4)
}
\end{pspicture}
\end{center}
\caption{}
\label{latticelema2}
\end{figure}
One checks with a computer that $\beta_{14}(I_{L_{AB}})\neq 0$ which implies the desired conclusion. 
\end{proof}

\begin{proof}[Proof of Theorem~\ref{linrel}]
Lemma~\ref{lema0}, Lemma~\ref{lema1}, and Lemma~\ref{lema2} complete the "only if" part of the proof. 

It remains to prove that if $L$ satisfies conditions (i) and (ii) in the statement of the theorem, then $I_L$ is linearly related.

To begin with, we recall from \cite{Q} that the ring $S/I_L$ which is isomorphic to $K[H]$ may be viewed as a semigroup ring  $K[s_1,
\ldots, s_m,t_1,\ldots,t_n]$ generated by the monomials $u_{ij}=s_it_j$ where $(i,j)\in L.$ With this interpretation of $K[H]$ in mind, 
we will use Corollary~\ref{refinement} to show that $I_L$ is linearly related.

Let $u=u_{i_1j_1}u_{i_2j_2}u_{i_3j_3}u_{i_4j_4}$ be an element of $K[H]$ viewed as a subring of the polynomial ring $K[s_1,
\ldots, s_m,t_1,\ldots,t_n]$  and let
$
i=\min_{q}\{i_q\}, k=\max_{q}\{i_q\}, j=\min_{q}\{j_q\}, \text{ and } \ell=\max_{q}\{j_q\}.
$
Therefore, all the points $h_q$ lie in the (possible degenerate) rectangle $\MQ$ of vertices
$(i,j), (k,j), (i,\ell),(k,\ell).$ If $\MQ$ is degenerate, that is, all the vertices of $Q$ are contained in a vertical or horizontal line segment in $L$, then $\beta_{1h}(I_{L})=0$ since in this case the simplicial complex $\Delta_h$ is just a simplex. Let us now consider $\MQ$ non-degenerate. If all the vertices of $\MQ$ belong to $L$, then the interval $L^\prime=[(i,j),(k,\ell)]$  is an induced sublattice  of $L$. Therefore, by Corollary~\ref{refinement}, we have
$\beta_{1h}(I_{L})=\beta_{1h}(I_{L^\prime})=0$, the latter equality being true since $L^\prime$ is linearly related. The only case left to be discussed 
is that one when one corner of the rectangle $\MQ$ does not belong to $L.$ In this case one, one esaily sees that, by Corollary~\ref{refinement}, 
$\beta_{14}(I_L)$ coincides with $\beta_{14}(I_{L^\prime})$ where $L^\prime$ is an induced sublattice of $L$ isomorphic to one displayed in Figure~\ref{caseone}.

\begin{figure}[hbt]
\begin{center}
\psset{unit=0.6cm}
\begin{pspicture}(4.5,-6)(4.5,3)

\rput(-4,0.5){
\pspolygon(2.8,-1)(2.8,0)(3.8,0)(3.8,-1)
\pspolygon(2.8,0)(2.8,1)(3.8,1)(3.8,0)
\pspolygon(3.8,-1)(3.8,0)(4.8,0)(4.8,-1)
\pspolygon(5.8,-1)(5.8,0)(4.8,0)(4.8,-1)
\pspolygon(4.8,0)(4.8,1)(5.8,1)(5.8,0)
\pspolygon(4.8,1)(4.8,2)(5.8,2)(5.8,1)
\pspolygon(3.8,1)(3.8,2)(4.8,2)(4.8,1)
\pspolygon(3.8,0)(3.8,1)(4.8,1)(4.8,0)
}
\rput(0.3,-1){(a)}

\rput(4.5,0.5){
\pspolygon(2.8,-1)(2.8,0)(3.8,0)(3.8,-1)
\pspolygon(2.8,0)(2.8,1)(3.8,1)(3.8,0)
\pspolygon(3.8,-1)(3.8,0)(4.8,0)(4.8,-1)
\pspolygon(5.8,-1)(5.8,0)(4.8,0)(4.8,-1)
\pspolygon(4.8,0)(4.8,1)(5.8,1)(5.8,0)
\pspolygon(4.8,1)(4.8,2)(5.8,2)(5.8,1)
\pspolygon(3.8,0)(3.8,1)(4.8,1)(4.8,0)
}
\rput(8.8,-1){(b)}

\rput(0.5,-4.5){
\pspolygon(2.8,-1)(2.8,0)(3.8,0)(3.8,-1)
\pspolygon(4.8,0)(4.8,1)(5.8,1)(5.8,0)
\pspolygon(3.8,-1)(3.8,0)(4.8,0)(4.8,-1)
\pspolygon(5.8,-1)(5.8,0)(4.8,0)(4.8,-1)
\pspolygon(4.8,1)(4.8,2)(5.8,2)(5.8,1)
\pspolygon(3.8,0)(3.8,1)(4.8,1)(4.8,0)
}
\rput(4.8,-6){(c)}
\end{pspicture}
\end{center}
\caption{}\label{caseone}
\end{figure}
\end{proof}

One may check with a computer algebra system that all lattices displayed in Figure~\ref{caseone} are linearly related, hence they do not have any relation in degree $h.$ Just one final word for $m=2$. In this case, we find an indiced sublattice of $L$ isomorphic to an induced sublattice of $L^\prime$, hence, again, we do not find any relation of $I_L$ in degree $4.$

\subsection{Hibi ideals with pure resolutions}
\label{pure}

Let $L$ be a planar distributive lattice, $L\subset [(0,0),(m,n)]$ with $m,n\geq 1.$. As in the previous subsection, we assume that the squares 
$[(0,0),(1,1)]$ and $[(m-1,n-1),(m,n)]$ belong to $L.$ In the last part of this section we would like find under which conditions on $L$ the ideal 
$I_L$ has a pure resolution. 

By Corollary~\ref{linear resolution}, we know that $I_L$ has a linear resolution if and only if $m=1$ or $n=1$. 
Therefore, we may consider $m,n\geq 2.$ 

We have already seen in Subsection~\ref{regsubsection} that if $C$ is a cyclic lattice, then $I_C$ has a pure 
resolution given by the Koszul complex of the sequence of its binomial generators. 
In addition, let us observe that if $L$ is not simple, then $R[L]$ may be expressed as $R[L]\cong R[L_1]\otimes R[L_2]$ where $L_1=\MI(P_1)$ and $L_2=\MI(P_2)$ with $P_1,P_2$ as they have been defined in the proof of Lemma~\ref{tensor}. Hence, if at least one of the ideals $I_{L_1}$ or $I_{L_2}$ has linear relations, then $I_L$ does not have a pure resolution since we have at least two distinct shifts in degree $1$ for $I_L.$ Therefore, from now on, we may assume that $L$ is a simple lattice.

If $L$ is not cyclic, then, by removing appropriate rows and columns of $L,$ we get an induced sublattice of $L$ of the form displayed in Figure~\ref{nonpure}.
\begin{figure}[hbt]
\begin{center}
\psset{unit=0.6cm}
\begin{pspicture}(-3,0)(7,4)

\rput(0,0){$\bullet$}
\rput(2,0){$\bullet$}
\rput(0,2){$\bullet$}
\rput(2,2){$\bullet$}
\rput(4,0){$\bullet$}
\rput(4,2){$\bullet$}
\rput(4,4){$\bullet$}
\rput(2,4){$\bullet$}
\pspolygon(0,0)(4,0)(4,2)(0,2)
\pspolygon(2,0)(4,0)(4,4)(2,4)
\end{pspicture}
\end{center}
\caption{}\label{nonpure}
\end{figure}

The resolution of $I_{L^\prime}$, where $L^\prime$ is the lattice of Figure~\ref{nonpure}, is the following:
\[
0\to S(-5)\to S(-3)^5\to S(-2)^5\to I_{L^\prime}\to 0.
\]

Hence, if $L$ is not cyclic, then $\beta_{13}(I_L)\neq 0.$ This implies that if $I_L$ has a pure resolution, then $I_L$ must be  linearly related, hence the lattice $L$ has the shape indicated in Theorem~\ref{linrel}.

Now we state the main result of this subsection.

\begin{Theorem}\label{pureres}
Let $L$ be a simple planar distributive lattice. Then $I_L$ has a pure resolution if and only if one of the following conditions holds:
\begin{itemize}
	\item [(i)] $L=\MI(P)$ where $P$ consists of a chain and an isolated vertex;
	\item [(ii)] $L$ is a cyclic lattice;
	\item [(iii)] $L$ is isomorphic either to the lattice displayed in Figure~\ref{nonpure} or to that one displayed in Figure~\ref{caseiii}.
\end{itemize}
\end{Theorem}

\begin{figure}[hbt]
\begin{center}
\psset{unit=0.6cm}
\begin{pspicture}(-3,0)(7,4)

\rput(0,0){$\bullet$}
\rput(2,0){$\bullet$}
\rput(0,2){$\bullet$}
\rput(2,2){$\bullet$}
\rput(0,4){$\bullet$}
\rput(2,4){$\bullet$}
\rput(4,4){$\bullet$}
\rput(2,4){$\bullet$}
\rput(4,0){$\bullet$}
\rput(4,2){$\bullet$}
\pspolygon(0,0)(4,0)(4,4)(0,4)
\psline(0,2)(4,2)
\psline(2,0)(2,4)
\end{pspicture}
\end{center}
\caption{}\label{caseiii}
\end{figure}

\begin{proof}
The "if" part is already clear since one may check with a computer that the idea of the lattice pictured in Figure~\ref{caseiii} has a pure resolution. For the converse, let us consider a simple planar distributive lattice $L\subset [(0,0),(m,n)]$ such that 
$I_L$ has a linear resolution. 

If $m=1$ or $n=1$, then $L$ satisfies condition (i). 

Let $m,n\geq 2$ and assume that $L$ is not cyclic. We have to show that $L$ satisfies condition (iii). By the arguments given before the theorem, we know that $L$ must satisfy the conditions of 
Theorem~\ref{linrel}. 

If $L$ is the whole interval $[(0,0),(m,n)]$ and $m\geq 3$ or $n\geq 3$, then we may obtain an induced sublattice
isomorphic to the lattice displayed in Figure~\ref{case3} which has the property that $I_L$ does not have a pure resolution. This check can be done by using a computer.

\begin{figure}[hbt]
\begin{center}
\psset{unit=0.6cm}
\begin{pspicture}(-1,0)(7,4)

\rput(0,0){$\bullet$}
\rput(2,0){$\bullet$}
\rput(0,2){$\bullet$}
\rput(2,2){$\bullet$}
\rput(0,4){$\bullet$}
\rput(2,4){$\bullet$}
\rput(4,4){$\bullet$}
\rput(2,4){$\bullet$}
\rput(6,0){$\bullet$}
\rput(6,2){$\bullet$}
\rput(6,4){$\bullet$}
\rput(4,0){$\bullet$}
\rput(4,2){$\bullet$}
\pspolygon(0,0)(6,0)(6,4)(0,4)
\psline(0,2)(6,2)
\psline(2,0)(2,4)
\psline(4,0)(4,4)
\end{pspicture}
\end{center}
\caption{}\label{case3}
\end{figure}

Therefore, in this case we get $m=n=2$ and $L$ is the lattice given in Figure~\ref{caseiii}. 

Let us now suppose that $L$ does not contain the vertex $(0,n).$ Then, by Theorem~\ref{linrel}, $L$ contains the vertex $(1,n-1).$ If $m\geq 3$ or $n\geq 3$, then, by removing suitable rows and columns of $L$ we get an induced sublattice  isomorphic to one of those pictured in  Figure~\ref{cornernonpure}.

\begin{figure}[hbt]
\begin{center}
\psset{unit=0.8cm}
\begin{pspicture}(4.5,0)(4.5,2.5)

\rput(-4,0.5){
\rput(2.8,0){$\bullet$}
\rput(2.8,1){$\bullet$}
\rput(3.8,0){$\bullet$}
\rput(3.8,1){$\bullet$}
\rput(4.8,0){$\bullet$}
\rput(4.8,1){$\bullet$}
\rput(5.8,0){$\bullet$}
\rput(5.8,1){$\bullet$}
\rput(4.8,2){$\bullet$}
\rput(5.8,2){$\bullet$}
\rput(3.8,2){$\bullet$}
\pspolygon(2.8,0)(2.8,1)(3.8,1)(3.8,0)
\pspolygon(4.8,0)(4.8,1)(5.8,1)(5.8,0)
\pspolygon(4.8,1)(4.8,2)(5.8,2)(5.8,1)
\pspolygon(3.8,1)(3.8,2)(4.8,2)(4.8,1)
\pspolygon(3.8,0)(3.8,1)(4.8,1)(4.8,0)
}
\rput(0.3,-0.6){(a)}

\rput(4.5,0.5){
\rput(2.8,0){$\bullet$}
\rput(2.8,1){$\bullet$}
\rput(3.8,0){$\bullet$}
\rput(3.8,1){$\bullet$}
\rput(4.8,0){$\bullet$}
\rput(4.8,1){$\bullet$}
\rput(5.8,0){$\bullet$}
\rput(5.8,1){$\bullet$}
\rput(4.8,2){$\bullet$}
\rput(5.8,2){$\bullet$}
\pspolygon(2.8,0)(2.8,1)(3.8,1)(3.8,0)
\pspolygon(4.8,0)(4.8,1)(5.8,1)(5.8,0)
\pspolygon(4.8,1)(4.8,2)(5.8,2)(5.8,1)
\pspolygon(3.8,0)(3.8,1)(4.8,1)(4.8,0)
}
\rput(8.8,-0.6){(b)}
\end{pspicture}
\end{center}
\caption{}\label{cornernonpure}
\end{figure}

None of the lattice displayed above has an ideal with pure resolution as one may check with the computer. Hence, $I_L$ itself does not have a pure resolution. Therefore, in this last case, if $I_L$ has a pure resolution, then $L$ must  be isomorphic to the lattice displayed in Figure~\ref{caseiii}.

\end{proof}

\end{document}